\documentclass[11pt, oneside]{article}   	

\usepackage[margin=2cm]{geometry}                		
\usepackage{graphicx,xspace}				

\usepackage{authblk}

\let\oldabstract\abstract
\makeatletter
\renewcommand\abstract{%
  \providecommand\keywords{\par\medskip\noindent\textbf{Keywords:}\xspace}
  \oldabstract\noindent\ignorespaces}
\makeatother

\newcommand{\rev}[1]{{#1}}
\newcommand{\reva}[1]{{#1}}
\newcommand{\revb}[1]{{#1}}

\usepackage{amsmath,amssymb}
\usepackage{amsthm}
\usepackage{cleveref}
\usepackage{xcolor}
\usepackage{subcaption}

\newcommand{\Nc}{\mathcal{N}}
\newcommand{\Xc}{\mathcal{X}}

\newcommand{\Ebb}{\mathbb{E}}
\newcommand{\Nbb}{\mathbb{N}}
\newcommand{\Pbb}{\mathbb{P}}
\newcommand{\Rbb}{\mathbb{R}}

\newcommand{\A}{\mathbf{A}}
\newcommand{\B}{\mathbf{B}}

\newcommand{\G}{\mathbf{G}}
\renewcommand{\H}{\mathbf{H}}
\newcommand{\I}{\mathbf{I}}

\newcommand{\X}{\mathbf{X}}

\renewcommand{\a}{\mathbf{a}}
\renewcommand{\b}{\mathbf{b}}
\renewcommand{\c}{\mathbf{c}}

\newcommand{\e}{\mathbf{e}}

\renewcommand{\v}{\mathbf{v}}
\newcommand{\x}{\mathbf{x}}
\newcommand{\y}{\mathbf{y}}
\newcommand{\z}{\mathbf{z}}

\newcommand{\varphib}{\boldsymbol{\varphi}}
\newcommand{\adj}{\mathrm{adj}}
\newcommand{\trace}{\mathrm{tr}}

\newcommand{\DPP}{\mathrm{DPP}}
\newcommand{\VS}{\mathrm{VS}}

\newenvironment{ProofOf}[1]{\begin{proof}[Proof of #1]}{\end{proof}}

\DeclareMathOperator*{\esssup}{ess\,sup}

\newtheorem{theorem}{Theorem}[section]
\newtheorem{lemma}[theorem]{Lemma}
\newtheorem{conjecture}[theorem]{Conjecture}
\newtheorem{example}[theorem]{Example}
\newtheorem{proposition}[theorem]{Proposition}
\newtheorem{remark}[theorem]{Remark}
\newtheorem{corollary}[theorem]{Corollary}

\date{}							

\graphicspath{{.//}
{./figs/}
}

\begin{document}
\title{Weighted least-squares approximation with determinantal point processes and generalized volume sampling}

\author{\textbf{Anthony Nouy}}  
\author{\textbf{Bertrand Michel}}

\affil{Nantes Universit\'e, Centrale Nantes, \\
Laboratoire de Math\'ematiques Jean Leray, CNRS UMR 6629}
\maketitle

\begin{abstract}

We consider the problem of approximating a function from $L^2$ by an element of a given $m$-dimensional space $V_m$, associated with some feature map $\varphib$,  using  evaluations of the function at random points $x_1, \hdots,x_n$. 
After recalling some results on optimal weighted least-squares using independent and identically distributed points, we consider weighted least-squares using projection determinantal point processes (DPP) or  volume  sampling. These   distributions introduce dependence between the points that promotes diversity in the selected features $\varphib(x_i)$. We first provide a generalized version  of volume-rescaled sampling yielding quasi-optimality results in expectation with a number of samples $n = O(m\log(m))$, that means that the expected $L^2$ error is bounded by a constant times the best approximation error in $L^2$. Also, further assuming that  the function is in some normed vector space $H$ continuously embedded in $L^2$, we further prove that the approximation error in $L^2$ is almost surely bounded by the best approximation error  measured in the $H$-norm. This includes the cases of functions from $L^\infty$ or reproducing kernel Hilbert spaces.
Finally, we present an alternative strategy consisting in using independent repetitions  of  projection DPP (or volume sampling), yielding similar error bounds as with i.i.d. or volume sampling, but in practice with a much lower number of samples. Numerical experiments illustrate the performance of the different strategies. 


\end{abstract}

\section{Introduction}
We consider the problem of approximating a function $f$ by an element of a given $m$-dimensional space $V_m$ using point evaluations of the function.
The function is defined on a \revb{set $\Xc$} equipped with a \revb{positive} measure $\mu$ and the error is assessed in the natural norm in $L^2_\mu(\Xc)$  defined by
$$
\Vert f \Vert^2 = \int_\Xc \vert f(x) \vert^2 d\mu(x).
$$ 
\revb{$\Xc$ can, for example, be a subset of $\Rbb^d$ but more general Polish spaces can be considered as well.} 
The best approximation error that can be achieved by elements of $V_m$ is 
$$\inf_{v \in V_m} \Vert f - g \Vert = \Vert f - P_{V_m} f\Vert $$
where $P_{V_m}f$ is the orthogonal projection of  $f$ onto $V_m$.
 An approximation $\hat f_m$ can be obtained by a weighted least-squares projection of $f$
defined as the minimizer of 
 \begin{equation}
\min_{g\in V_m} \frac{1}{n}\sum_{i=1}^n \revb{w(x_i)} \vert f(x_i) - g(x_i) \vert^2 \label{eq:weighted-ls}
\end{equation}
where $w : \Xc \to \Rbb$ is a positive weight function and the  $x_1, \hdots,x_n$ are points in $\Xc$. 
The approximation $\hat f_m $ is said quasi-optimal if 
$$
\Vert f - \hat f_m \Vert \le C \inf_{g \in V_m} \Vert f - g \Vert,
$$
with a constant $C$ independent of $m.$ When using random points, it is said   
quasi-optimal in expectation whenever 
$$
\Ebb(\Vert f - \hat f_m \Vert^2)^{1/2} \le C \inf_{g \in V_m} \Vert f - g \Vert,
$$
\reva{which guarantees that the averaged error $\Ebb(\Vert f - \hat f_m \Vert^2)^{1/2} $ converges as least as fast as the best approximation error $e_m(f)_{L^2} :=  \inf_{g \in V_m} \Vert f - g \Vert$}.
A fundamental problem is to select points and weights that achieve quasi-optimality with a number of points as close as possible to  the dimension $m$ of $V_m.$ 
 The weighted least-squares approximation $\hat f_m$ defined by \eqref{eq:weighted-ls} is such that 
   \begin{equation}
\Vert f - \hat f_m \Vert_n = \min_{g\in V_m} \Vert f - g \Vert_n
 \label{eq:weighted-ls-seminorm}
\end{equation}
where $\Vert \cdot  \Vert_n$ is the empirical \rev{(discrete)} semi-norm defined by
 \begin{equation}\Vert  f \Vert_n^2 =  \frac{1}{n}  \sum_{i=1}^n \revb{w(x_i)} f(x_i)^2.\label{semi-norm}\end{equation} 
$\hat f_m $ is the orthogonal projection $\hat P_{V_m} f$ of $f$ onto $V_m$   with respect to the empirical semi-norm, 
 and the quality of the approximation is related to how close $\Vert  \cdot \Vert_n$  is from the  norm $\Vert \cdot \Vert$. 
 
 We assume that we are given an orthonormal basis $\varphi_1,\hdots , \varphi_m$ of $V_m$, and we let $\varphib : \Xc \to \Rbb^m $ be the associated feature map defined by $\varphib(x)  = (\varphi_1(x),\hdots,\varphi_m(x))^T$. Then for any $g$ in $V_m$,  where  $g(x) = \varphib(x)^T \a$ for some $\a\in \Rbb^m$, it holds 
 $
 \Vert g \Vert^2 = \Vert \a \Vert_2^2  $ {and} $\Vert g \Vert_n^2 = \a^T \G^w \a,
$ where $\G^w$ is the empirical Gram matrix 
\begin{equation}
\G^w = \G^w(x_1, \hdots , x_m) :=  \frac{1}{n}\sum_{i=1}^n \revb{w(x_i)} \varphib(x_i) \varphib(x_i)^T, \label{empirical-gram}
\end{equation}
so that
\begin{equation}
\lambda_{min}(\G^w) \Vert g \Vert^2 \le \Vert g \Vert_n^2 \le \lambda_{max}(\G^w) \Vert g \Vert^2, \quad \forall g \in V_m, \label{semi-norm-equiv}
\end{equation}
\rev{which is known as a Marcinkiewicz-Zygmund inequality in sampling discretization  \cite{Kashin2022Aug}. }
The quality of the projection is therefore  related to how much the spectrum of $\G^w$ deviates from one. In particular, it holds 
$$
\Vert f - \hat f_m \Vert^2 \le \Vert f - P_{V_m} f \Vert^2 + \lambda_{min}(\G^w)^{-1} \Vert f - P_{V_m} f \Vert_n^2.
$$  
A control  of the minimal eigenvalue of $\G^w$ is therefore necessary to achieve quasi-optimality. A control  of the highest eigenvalue of $\G^w$ is also needed for numerical stability reasons, so that quasi-optimality can be achieved in finite precision arithmetic. 
The choice of  optimal points (and weights) is  a classical problem of design of experiments \cite{pukelsheim2006optimal}.  
 A classical approach, called $E$-optimal design, consists in selecting points (and weights) that maximize $\lambda_{min}(\G^w)$. Variants of this problem consist in maximizing the trace of the inverse of $\G^w$, which is called $A$-optimal design, or maximizing the determinant $\det(\G^w)$, which is  called $D$-optimal design. The latter is related to Fekete points for polynomial interpolation or more general kernel based interpolation \cite{Bos2010Nov,Karvonen2021May}. It is also related   
 to maximum volume concept in linear algebra \cite{goreinov2001maximal,fonarev2016efficient}.
However,   these optimization problems are in general intractable.  
 
The above mentioned approaches are deterministic. Here, we follow a probabilistic avenue, where the points $x_1,\hdots, x_n$ are drawn from a suitable distribution allowing a control of the spectrum of the empirical Gram matrix. 
When the points $x_i$ are drawn from a distribution $\nu$ with density $\revb{w^{-1}}$ with respect to $\mu$, the empirical Gram matrix is an unbiased estimate of the identity. Provided the points are independent and identically distributed (i.i.d.),  the empirical Gram matrix almost surely converges to the identity and matrix concentration inequalities allow to analyze how fast is this convergence. An optimization of the convergence rate  over all possible distributions yields an optimal density $\revb{w_m(x)^{-1}} = \frac{1}{m} \Vert \varphib(x) \Vert^2_2$, that is known as the inverse Christoffel function for polynomial spaces $V_m$ \cite{cohen2017:optimal_weighted_least_squares}. The measure $\nu_m  = \revb{w_m^{-1}} \mu$ is also  known as leverage score distribution in statistics and machine learning.  Sampling from this distribution guarantees  
  that the event $S_\delta = \{ \lambda_{min}(\G^{w_m}) \ge 1-\delta\}$ is satisfied with a controlled probability $1-\eta$ provided the number of samples $n = O( \delta^{-2}m\log(m\eta^{-1}))$, where the dependence in $m$ is known to be optimal for i.i.d. sampling \cite{Sonnleitner2023Oct}. A similar control in probability is obtained for the maximum eigenvalue, and $\lambda_{max}(\G^{w_m}) \le m$ even holds  almost surely with the particular choice of weight function $w_m$.  By drawing i.i.d. samples from $\nu_m$ and conditioning to  the event $S_\delta$ (that can be achieved by a rejection sampling with controlled rejection probability), it holds $\Ebb(\Vert \cdot \Vert_n^2) \le \beta \Vert \cdot \Vert^2$ for some constant $\beta$, and  the resulting least-squares projection $\hat f_m$ is quasi-optimal in expectation. If we further assume that the target function $f$ is in a subspace $H$ continuously embedded in $L^2_\mu(\Xc)$ and \rev{$L^\infty_{\mu,h^{-1/2}}(\Xc)$ (the space of functions $f$ defined on $\Xc$ such that $h^{-1/2} f$ is uniformly bounded)}, with $h$ a probability density with respect to $\mu$, and if we choose for $\nu = \revb{w^{-1}} \mu$ a mixture of the optimal sampling distribution $\nu_m = \revb{w_m^{-1}} \mu$  and $h \mu$, we    prove  that quasi-optimality still holds  in expectation and  we   also prove that $\Vert \cdot \Vert_n$ is almost surely bounded  by $\Vert \cdot \Vert_H$ (up to a constant), which ensures that it holds almost surely
   \begin{equation}
   \Vert f - \hat f_m \Vert \le C \inf_{g\in V_m} \Vert f - g \Vert_H, \label{L2-H-quasi-optimality}
   \end{equation}
   which we will call $H\to L^2_\mu$ quasi-optimality. Examples of such spaces $H$ are $L^\infty_\mu(\Xc)$ (with \revb{$\mu$ a finite measure and  $h=\mu(\Xc)^{-1}$}), or reproducing kernel Hilbert spaces. \revb{Note that the idea of using a mixture between $\nu_m$ and $\mu$ to control the discrete norm by the $L^\infty_\mu$-norm is not new, see, e.g., \cite{Pozharska2022Jun,bartel2023}.} \reva{The inequality \eqref{L2-H-quasi-optimality} ensures that the approximation error in $L^2$-norm is upper bounded by the 
best approximation error in $H$-norm  $e_m(f)_H := \inf_{g\in V_m} \Vert f - g \Vert_H$.  Of course, further assumptions on $f$ and a suitable choice of $V_m$ are required to guarantee some decay of  $e_m(f)_H $ (which converges slower than $e_m(f)_{L^2}$ in general). In this paper, we are not concerned with the choice of $V_m$ (which is assumed to be given) and the analysis 
of the convergence of best approximation errors in $V_m$ (in $L^2$ or $H$ norms), but only with the construction of algorithms yielding quasi-optimal approximations (in expectation, with high probability or almost surely).  
}  
  
In practice, the number of i.i.d. samples $n$ needed for  a stable projection may be large  and far from the dimension $m$.  
In order to further reduce the sampling complexity, various subsampling approaches have been recently proposed. They start with a set of points that guarantee that the spectrum of $\G^w$ is contained  in some interval $[a,b]$, 
and then extract a subset of points that  guarantee that the spectrum of the empirical Gram matrix (up to a possible reweighting)  is still contained in some prescribed interval $[
a',b']$. The approach proposed in 
  \cite{dolbeault2022optimalL2,dolbeault2023sharp} yields quasi-optimality in expectation with a number of samples $n =O( m)$. The algorithm is a randomized version of algorithms provided in \cite{marcus2015interlacing,nitzan2016exponential} for the solution of the Kadinson-Singer problem. This algorithm is unfortunately intractable. However, it is interesting from a theoretical perspective since it allows to prove that quasi-optimality in expectation can be achieved with a number of samples linear in $m$, therefore showing that  sampling numbers in a randomized setting and Kolmogorov widths are comparable for compact sets in $L^2_\mu(\Xc)$. 
A greedy subsampling algorithm with polynomial complexity has been proposed in \cite{haberstich2022boosted}, that  reaches in practice a number of samples $n$ close (and sometimes equal) to $m$. However, it provides a suboptimal guaranty in expectation, that is $\Ebb(\Vert f - \hat f_m \Vert^2)^{1/2} \le C \log(m)^{1/2} \Vert f - P_{V_m} f\Vert $, and no theoretical guaranty to extract a set of samples of size $n = O( m)$. Another tractable approach has been proposed in \cite{bartel2023}, which allows to reach a number of samples $n = O( m)$. Yet, this algorithm does not provide quasi-optimality in expectation. These conditioning and subsampling approaches all yield a set of points with a dependence structure that is not given explicitly. They require to start with a rather large set of samples and suffer from the complexity of  subsampling algorithms, which is polynomial in the initial number of samples.  

Another route is to leave the i.i.d. setting from the start and sample from a  distribution that introduces a dependence between the samples. \revb{An algorithm which achieves quasi-optimality in expectation with $n=O(m)$ samples has been proposed in \cite{Dolbeault2024Jul}. It is a randomized variant of the subsampling algorithm from \cite{bartel2023}. Another}
 prominent  approach is to rely on volume sampling, first introduced in a discrete setting in  \cite{deshpande2006matrix,avron2013faster}, and then extended to more general settings in  \cite{Poinas:2022aa}. Volume sampling has found many applications in machine learning. 
For classical (non weighted) least-squares, it consists in drawing samples $\x = (x_1, \hdots, x_n)$ from a distribution $\gamma_n$  over $\Xc^n$ having a density proportional to $\det(\Phi(\x)^T \Phi(\x))$, with $\Phi(\x) = (\varphib(x_1) , \hdots , \varphib(x_n))^T.$ The distribution $\gamma_m$, for $n=m$, corresponds to a projection determinantal point processes (DPP) \cite{lavancier2015dpp}. The density drops down to zero whenever two vectors $\varphib(x_i)$ get collinear, hence this distribution introduces a repulsion between the points and promotes diversity in the selected features $\varphib(x_i)$.  
For $n>m$, up to a random permutation of points, this distribution corresponds to $m$ points from a projection DPP and an independent set of $n-m$ i.i.d. samples from $\mu$ (\revb{provided $\mu$ is a probability measure}).
The associated empirical Gram matrix (with weight $w=1$) has bad concentration properties. 
Here we consider a generalized volume sampling distribution $\gamma_n^\nu$ for weighted least-squares, which has a density proportional to $\det(\G^w(\x))$ with respect to a product measure $\nu^{\otimes n}$\revb{ (the  measure $\mu$ is no more required to be a probability measure)}. This introduces a compromise between promoting a high likelihood with respect to the reference measure $\nu$ and promoting a high determinant of the empirical Gram matrix. For $\nu = \nu_m$, $\gamma^{\nu_m}_n$ corresponds to the  volume-rescaled sampling distribution introduced in \cite{derezinski2022unbiased}.  This distribution yields quasi-optimality in expectation, without the need of conditioning. Moreover, this distribution has the very nice property of providing an unbiased approximation, i.e. $\Ebb(\hat f_m) = P_{V_m} f$, which allows to perform an averaging of estimators for improving quasi-optimality constant. 

Our first main contribution is to consider a general version of volume-rescaled sampling distribution, with a measure   $\nu = \revb{{w^{-1}}} \mu$  allowing to obtain not only quasi-optimality in expectation but also an almost sure $H\to L^2_\mu$ quasi-optimality for functions from subspaces $H$ described above. 
Despite the many advantages of volume sampling compared to i.i.d. optimal sampling, the  number of samples to ensure stability of the empirical Gram matrix with high probability is essentially of the same order as for i.i.d. sampling, i.e. $n = O(m\delta^{-2}\log(m\eta^{-1}))$. 

Our second contribution is to propose an alternative that consists in using $r$ independent samples from the projection DPP distribution  $\gamma_m$, or from the volume sampling distribution  $\gamma_n^\nu$ with a suitable mixture distribution $\nu$.  Using conditioning, the former allows to obtain quasi-optimality in expectation, while the latter allows to achieve $H \to L^2_\mu$ quasi-optimality almost surely for a subspace of functions $H$. These results are similar to optimal i.i.d. sampling (with suitable mixture measures) or to our general version of volume sampling. We can prove that stability $S_\delta$ is achieved with probability $1-\eta$ under a suboptimal condition $n = O(m^2\delta^{-2}\log(m\eta^{-1}))$, or a better condition $n = O(m\delta^{-2}\log(m\eta^{-1}))$ (similar to i.i.d. and volume sampling) under a conjecture (only checked numerically) on the properties of the empirical Gram matrix associated with projection DPP or volume sampling. Although this theoretical guaranty does not show any advantage of this new sampling strategy, we observe in practice a much better concentration of the empirical Gram matrix, hence a much lower number of samples needed for obtaining the stability condition $S_\delta$ with high probability.
 
 \revb{Although they are not directly in line with our setting (the approximation of a function in an arbitrary subspace $V_m$), we would like to mention  related works \cite{Belhadji2020Nov,Belhadji2023Oct} using  determinantal point processes for the approximation of functions from reproducing kernel Hilbert spaces $H$. In these works, the sampling distribution is related to the kernel of $H$.}
 
 \reva{In this paper, we only provide upper bounds of the error in $L^2$-norm in terms of errors of best approximation in $L^2$ or $H$ norms. Obtaining a control of the error in other norms, e.g. $L^\infty$ or a some RKHS norm, would certainly be of interest but this in general requires to modify  the projection or the sampling methods, see the recent works \cite{Krieg2024Jan,trunschke2024almostsurequasioptimalapproximationreproducing} in this direction.}
 
 The outline of the paper is as follows. In \Cref{sec:preliminary}, we provide some preliminary results on weighted least-squares projections. In \Cref{sec:iid}, we recall some classical results on optimal weighted least-squares with i.i.d. sampling, with quasi-optimality results in expectation and also $H\to L^2_\mu$ quasi-optimality results for a large class of function spaces, that extend previous results \cite{haberstich2022boosted} to a more general setting. In \Cref{sec:dpp}, we introduce DPP and more general volume sampling distributions, and analyze the properties of corresponding weighted least-squares projections. In particular we obtain quasi-optimality in expectation and almost sure $H\to L^2_\mu$ quasi-optimality when using our general volume sampling distribution with a suitable weight function.  
In \Cref{sec:repeated}, we present the alternative strategy consisting in using independent repetitions of DPP (or volume sampling), and obtain similar quasi-optimality results.
In \Cref{sec:experiments}, we provide numerical evidence of the efficiency of the strategy based on independent repetitions of DPP, compared to optimal i.i.d. or volume-rescaled sampling.  
 
 
%

\section{Preliminary results on weighted least-squares approximation}
\label{sec:preliminary}
Here, we provide some preliminary results on weighted least-squares approximation. We start with a control of the \reva{bias} of the empirical semi-norm, provided a condition on the weight function $w$ that needs to be related to the sampling distribution.
 \begin{lemma}\label{lem:semi-norm-bound}
Assuming that the points are drawn from a distribution over $\Xc^n$ with marginals all equal to $\tilde \nu = \tilde w\mu$, and assuming that the weight function $w$ is such that  $\revb{ w \le \beta \tilde w}$, it holds for all $f \in L^2_\mu$
  $$
\Ebb(\Vert f \Vert^2_n) \le \beta  \Vert f \Vert^2.
$$
with equality when $\tilde w = w.$ 
\end{lemma}
\begin{proof}
It holds $\Ebb(\Vert f \Vert^2_n) =  \frac{1}{n}\sum_{i=1}^n \Ebb_{x_i \sim \tilde \nu} (\revb{w(x_i)} f(x_i)^2 )
\le      \beta \revb{ \int f(x)^2 d\mu(x)}  = \beta  \Vert f \Vert^2.$
\end{proof}

 Assuming $\G^w$ invertible, the projection error satisfies 
\begin{align}
\Vert f - \hat f_m \Vert^2 &= \Vert f - P_{V_m} f \Vert^2 +  \Vert \hat P_{V_m} (f - P_{V_m} f ) \Vert^2 \label{split-error-orthogonal}
\end{align}
from which we deduce 
\begin{align*}
\Vert f - \hat f_m\Vert^2 \le  \Vert f - P_{V_m} f \Vert^2 + \lambda_{min}(\G^w)^{-1} \Vert  f - P_{V_m} f  \Vert_n^2,
\end{align*}
and the following result. 

\begin{lemma}\label{lem:quasi-optimality-conditional-expectation}
Assume that the points $(x_1, \hdots, x_n)$ are drawn from a distribution over $\Xc^n$ with marginals all equal to $\tilde \nu = \revb{\tilde w^{-1}} \mu$ and we use weighted least-squares with a weight function $w$ such that $\revb{ w \le \beta \tilde w}$. Letting $S_\delta = \{\lambda_{min}(\G^w) \ge 1-\delta\}$, it holds for any $f \in L^2_\mu$
$$
\Ebb(\Vert \hat P_{V_m} f \Vert^2 \vert S_\delta) \le  \Pbb(S_{\delta})^{-1} (1-\delta)^{-1}  \beta \Vert f \Vert^2,
$$
and 
$$
\Ebb(\Vert f -  \hat P_{V_m} f \Vert^2 \vert S_\delta ) \le  (1+ \Pbb(S_\delta)^{-1} (1-\delta)^{-1} \beta)  \inf_{g\in V_m} \Vert f - g \Vert^2.
$$
\end{lemma}
\begin{proof}
From \eqref{semi-norm-equiv}, we obtain $\Ebb(\Vert \hat P_{V_m} f \Vert^2 \vert S_\delta) \le (1-\delta)^{-1} \Ebb(\Vert \hat P_{V_m} f \Vert_n^2 \vert S_\delta) \le (1-\delta)^{-1}   \Pbb(S_\delta)^{-1} \Ebb(\Vert \hat P_{V_m} f \Vert_n^2  ),$ and since $\hat P_{V_m}$ is an orthogonal projection with respect to the inner product $\Vert \cdot\Vert_n$, it holds 
$\Ebb(\Vert \hat P_{V_m} f \Vert_n^2  ) \le \Ebb(\Vert   f \Vert_n^2  ) \le \beta \Vert f \Vert^2 $, where the last inequality results from Lemma \ref{lem:semi-norm-bound}. 
The second inequality then follows from \eqref{split-error-orthogonal}.
\end{proof}
In order to obtain error bounds with high probability or even almost surely, we introduce additional assumptions on the target function, and choose  a weight function accordingly.  \rev{For some strictly positive  function $h$, we let $L^\infty_{\mu,h^{-1/2}}(\Xc)$ be the space of functions defined on $\Xc$ such that $f h^{-1/2}$ is in $L^\infty_\mu(\Xc)$.}
Let $H$ be a normed vector space of functions defined on $\Xc$,   
continuously embedded in both $L^2_\mu$ and $L^\infty_{\rev{\mu,h^{-1/2}}}$. That means respectively that for any $f\in H$,
\begin{equation}\Vert f \Vert \le C_H \Vert f \Vert_H,
\label{H-embedding}
\end{equation}
and  
\begin{equation}
\Vert f \Vert_{L^\infty_{\mu,h^{-1/2} }} = \esssup_{x\in \Xc} h(x)^{-1/2} \vert f(x) \vert \le \Vert f \Vert_H
\label{H-hcondition}.
\end{equation}

\begin{example}\label{example:Linf}
\revb{When $\mu$ is a probability measure,}
the properties \eqref{H-embedding} and \eqref{H-hcondition}
hold for  $H = L^\infty_\mu$, with $h=1$, and embedding constant $C_H = 1$. 
\end{example}

\begin{example}\label{example:rkhs}
The properties \eqref{H-embedding} and \eqref{H-hcondition}
hold for $H$  a reproducing kernel Hilbert space of functions with kernel $K : \Xc\times \Xc\to \Rbb$ having finite trace $\int K(x,x) d\mu(x) <\infty$. $H$ is compactly embedded in $L^2_\mu$ with embedding constant $C_H^2 = \int K(x,x) d\mu(x)$, and continuously embedded in \rev{$L^\infty_{\mu,h^{-1/2} }$} with  $h(x) = K(x,x)$. The kernel admits a Mercer decomposition $K(x,x) = \sum_{i=1}^M \lambda_i \psi_i(x) \psi_j(y)$ with $M\in \Nbb \cup \{+\infty\}$,   where the $\psi_i$ form an orthonormal system in $L^2_\mu$ and the $\lambda_i>0$ are such that $\sum_{i=1}^M \lambda_i = C_H^2.$ 
The kernel can be rescaled such that $C_H = 1$, in which case $h$ is a \rev{probability} density with respect to $\mu$.
In the case when \revb{$\mu$ is itself a probability measure and} $\Xc$ has a group structure with $K(x,y) = k(x-y)$, then $h(x) = k(0) $ is a constant function, and with the previously mentioned rescaling, $C_H = 1$ and $h =1$, \revb{and $H$ is continuously embedded in $L^\infty_\mu$. More generally, when $h$ is uniformly bounded, $H$ is continuously embedded in $L^\infty_\mu$. However, there are some interesting cases of RKHS for which $h$ is not uniformly bounded, e.g. the  Sobolev space $H^1_\nu(\mathbb{R})$ with $\nu = \mathcal{N}(0,1)$ the standard Gaussian measure, whose kernel has diagonal $h(x) = \sqrt{\pi/2} \exp(x^2)(1-\mathrm{erf}(x/\sqrt{2})^2)$, and which is continuously embedded in $L^2_\mu$ for $\mu \sim \Nc(0,a)$, for $a<1$.}  We refer to \cite{berlinet2011reproducing}  for an introduction to RKHS.
\end{example}
Noting that for any $g\in V_m$, it holds
$$
\Vert f - \hat f_m \Vert \le \Vert f - g \Vert + \lambda_{min}(\G^w)^{-1/2} \Vert f - g \Vert_n,
$$
we can deduce another useful lemma, provided some condition on the sampling measure.  

\begin{lemma}\label{lem:near-optimality-H}
Assume $f \in H$ with $H$ satisfying \eqref{H-embedding} and $\eqref{H-hcondition}$. If the weight function $w$ is such that   $w \ge \zeta^{-1} h$, it holds  $\Vert g \Vert_n^2 \le \zeta \Vert g \Vert_H^2$ almost surely, for any $g \in H$. 
If we further assume that  $\G^w$ is almost surely invertible, it holds  almost surely
\begin{align*}
\Vert f - \hat f_m \Vert &\le \left( C_H+ \lambda_{min}(\G^w)^{-1/2} \zeta^{1/2} \right) \inf_{g\in V_m}\Vert f - g \Vert_H .
\end{align*}
\end{lemma}


\section{Least-squares with independent and identically distributed samples}
\label{sec:iid}

We here consider the classical setting where the $x_1,\hdots,x_n$ are i.i.d. samples from a distribution $\nu = \revb{w^{-1}} \mu$ with density $\revb{w^{-1}}$ with respect to $\mu$. 
The empirical Gram matrix can be written 
$$
\G^w = \frac{1}{n} \sum_{i=1}^n \A_i, \quad \A_i = \revb{w(x_i)} \varphib(x_i) \varphib(x_i)^T.
$$ 
The $\A_i$ are i.i.d. rank-one matrices, with expectation $\Ebb(\A_i) = \I$ and spectral norm satisfying almost surely
$$
\Vert \A_i \Vert  = \revb{w(x_i)}  \Vert \varphib(x_i) \Vert_2^2.
$$ 
From matrix Chernoff inequality (recalled in Theorem \ref{th:matrix-chernoff}), 
we then deduce the following result from  \cite{cohen2017:optimal_weighted_least_squares}. 
\begin{lemma}\label{lem:matrix-chernoff}
Assume the points $(x_1,\hdots,x_n)$ are i.i.d. samples from $\nu =  \revb{w^{-1}} \mu$, with $w$ such that  
$$
K_{w,m} = \sup_{x \in \Xc} \revb{w(x)} \Vert \varphib(x) \Vert_2^2 <\infty.
$$
Then for any   $0<\delta<1$, it holds 
$$
\Pbb( \lambda_{min}(\G^w) < 1 - \delta) \le m \exp(- n c_\delta / K_{w,m})
$$
with $c_\delta = \delta + (1-\delta) \log(1-\delta)$ such that  $
 \delta^2/2 \le  c_\delta \le  \delta^2$. Then it holds 
 $$
\Pbb(\lambda_{min}(\G^w) < 1-\delta) \le \eta  \quad \text{if} \quad n \ge c_\delta^{-1} K_{w,m} \log(m\eta^{-1}).
 $$
\end{lemma}
Since $$K_{w,m} \ge \mathbb{E}_{x\sim \nu } (\revb{w(x)} \Vert \varphib(x) \Vert_2^2) = \revb{\int \Vert \varphib(x) \Vert_2^2 d \mu(x)} = m,$$ we deduce that $K_{w,m} \ge m$.
The optimal sampling measure that minimizes the upper bound of the matrix Chernoff inequality  
 is therefore given by 
$$
\nu_m = \revb{w_m^{-1}} \mu,$$
where the density $\revb{w_m^{-1}} $ with respect to $\mu$ is given by
$$\revb{w_m^{-1}(x)} :=  \frac{1}{m} \sum_{i=1}^m \varphi_i(x)^2 =\frac{1}{m} \Vert \varphib(x)\Vert_2^2 ,
$$
that provides an optimal constant 
$K_{w_m,m} = m $.  This optimal distribution for i.i.d. sampling is also known as leverage score distribution. Choosing a function $w$ such that $\revb{w^{-1}} \ge \alpha \revb{w_m^{-1}}$ for some $\alpha>0$ yields a constant 
\begin{equation}
K_{w,m} \le \alpha^{-1} K_{w_m,m} = \alpha^{-1} m,\label{eq:boundKmixture}
\end{equation}
and we have $\Pbb(\lambda_{min}(\G^w) < 1-\delta) \le \eta$ provided $n \ge c_\delta^{-1} \alpha^{-1} m
 \log(m\eta^{-1})$.
\\\par 
We next provide a useful lemma on the stability of the empirical least-squares projection. 
\begin{lemma}\label{lem:empirical_projection_iid}
Assume that $(x_1,\hdots,x_n)$ is drawn from $\nu^{\otimes n}$ with $\nu = \revb{w^{-1}} \mu$. Let $S_\delta = \{\lambda_{min}(\G^w) \ge 1-\delta\}$ with $0< \delta<1.$
Then for any $f\in L^2_\mu$,  it holds
$$
\Ebb(\Vert \hat P_{V_m}  f\Vert^2 \vert S_\delta) \le \Pbb(S_\delta)^{-1} (1-\delta)^{-1}   \Vert f \Vert^2,
$$
and 
$$
\Ebb(\Vert \hat P_{V_m}  f\Vert^2 \vert S_\delta) \le \Pbb(S_\delta)^{-1} (1-\delta)^{-2} \left(\Vert P_{V_m} f \Vert^2 +  \frac{K_{w,m}}{n} \Vert f \Vert^2 \right).
$$
\end{lemma}
\begin{proof}
The first inequality directly comes from Lemma \ref{lem:quasi-optimality-conditional-expectation} with $\beta = 1.$ For the proof of the second inequality, let $\G := \G^w,$ and first note that  
  $\hat P_{V_m} f(x) = \varphib(x)^T \c $, with $\c = \G^{-1} \b$ and  $\b = \frac{1}{n} \sum_{i=1}^n  \revb{w(x_i)} \varphib(x_i) f(x_i)$. Therefore
$\Vert \hat P_{V_m} f \Vert^2 =  \Vert \c \Vert_2^2 =  \Vert \G^{-1}\b \Vert_2^2  \le \lambda_{min}(\G)^{-2} \Vert \b \Vert_2^2,$
and 
$
\Ebb(\Vert \hat P_{V_m} f \Vert^2 \vert S_\delta) \le (1-\delta)^{-2} \Pbb(S_\delta)^{-1} \Ebb(\Vert \b \Vert_2^2).
$
Then we have 
\begin{align*}
\Ebb(\Vert \b \Vert_2^2) &= \frac{1}{n^2} \sum_{i,j=1}^n \Ebb(\revb{w(x_i)}  \varphib(x_i)^T f(x_i) \revb{w(x_j)}  \varphib(x_j)^T f(x_j))\\
&= \frac{1}{n} \Ebb_{x \sim \nu}(f(x)^2 \revb{w(x)^2}  \Vert \varphib(x) \Vert^2_2) + \frac{n-1}{n} \revb{ \Vert \int f(x)  \varphib(x) d \mu(x)  \Vert_2^2 }   \\
& \le  \frac{K_{w,m}}{n} \Vert f \Vert^2 + \frac{n-1}{n} \Vert P_{V_m} f \Vert^2,
\end{align*}
which  ends the proof.
\end{proof}
\begin{theorem}\label{th:quasi-optimality-iid-expectation}
Assume that $(x_1,\hdots,x_n)$ is drawn from $\nu^{\otimes n}$ with $\nu = \revb{w^{-1}} \mu$ such that 
 $\revb{w^{-1}} \ge   \alpha \revb{w_m^{-1}} $ for some $\alpha >0.$ 
Further  assume that   
$$n \ge c_\delta^{-1}  \alpha^{-1} m \log(m\eta^{-1}),$$
with $0< \delta<1.$ Then the event $S_\delta = \{\lambda_{min}(\G^w) \ge 1-\delta\}$ is such that  
$\Pbb(S_\delta) \ge 1 - \eta$ and it holds 
$$
\Ebb(\Vert f - \hat f_m \Vert^2 \vert S_\delta) \le (1 + (1-\eta)^{-1} (1-\delta)^{-1} ) \inf_{g\in V_m} \Vert f - g \Vert^2,
$$
and 
$$
\Ebb(\Vert f - \hat f_m \Vert^2 \vert S_\delta) \le (1 +  \alpha^{-1} \frac{m}{n}    (1-\eta)^{-1} (1-\delta)^{-2} ) \inf_{g\in V_m} \Vert f - g \Vert^2,
$$
\end{theorem}
\begin{proof}
The first inequality comes from Lemma \ref{lem:quasi-optimality-conditional-expectation} and Lemma \ref{lem:matrix-chernoff}, while the second inequality follows from \eqref{split-error-orthogonal}, Lemma \ref{lem:matrix-chernoff}  and Lemma \ref{lem:empirical_projection_iid}, noting that 
 $P_{V_m} (f-P_{V_m} f) = 0$.
\end{proof}


%
The next theorem provides a control of error in probability, provided that the target function $f$ is in a space $H$ satisfying \eqref{H-embedding} and \eqref{H-hcondition}, with $h$ a \revb{probability} density w.r.t. $\mu$, and we use a sampling distribution 
$\nu = \revb{w^{-1}} \mu$ with  
\begin{align}
\revb{w^{-1}}  = \alpha \revb{w_m^{-1}} +  ({1-\alpha}) h. \label{eq:mixture_density}
\end{align}
The measure $\nu $ is  a mixture between $\nu_m = \revb{w_m^{-1}} \mu$ and the measure $h \mu$, with respective weights $\alpha$ and $1-\alpha.$ 
\begin{theorem}\label{th:quasi-optimality-iid-H}
Assume  $f\in H$, with $H$ satisfying \eqref{H-embedding} and \eqref{H-hcondition}, with $h$ a \revb{probability} density with respect to $\mu$. 
Assume  $(x_1,\hdots,x_n)$ is drawn from $\nu^{\otimes n}$ with $\nu = \revb{w^{-1}} \mu$ and $\revb{w^{-1}} =   \alpha \revb{w_m^{-1}} +  ({1-\alpha}) h$. Then provided 
$$n \ge c_\delta^{-1} \alpha^{-1} m\log(m\eta^{-1}),$$
with $0< \delta<1,$ the event $S_\delta = \{\lambda_{min}(\G^w) \ge 1-\delta\}$ is such that  
$\Pbb(S_\delta) \ge 1 - \eta$ and 
 it holds 
\begin{align*}
\Vert f - \hat f_m \Vert &\le \left( C_H+ (1-\delta)^{-1/2} (1-\alpha)^{-1/2} \right) \inf_{g\in V_m}\Vert f - g \Vert_H  
\end{align*}
with probability greater than $1-\eta.$
\end{theorem}
\begin{proof}
It is directly deduced from Lemma \ref{lem:near-optimality-H} and Lemma \ref{lem:matrix-chernoff}.
\end{proof}
\begin{remark}[Sampling from the mixture $\nu$]
Sampling from the mixture $\nu = \alpha \nu_m + (1-\alpha) h\mu$ can be performed by sampling from $\nu_m$ with probability $\alpha$ and from $h\mu$ with probability $1-\alpha$. When $H=L^\infty_\mu$ \revb{and $\mu$ is a probability measure} (see Example \ref{example:Linf}), $h=1$ and it requires sampling from the reference measure $\mu.$ 
When $H$ is a RKHS (see Example \ref{example:rkhs}) with kernel $K$ such that $\int K(x,x) d\mu(x) = 1$, it requires sampling from $h \mu$ with $h(x) = K(x,x)$. If $K$ is known from its  Mercer decomposition $K(x,y) = \sum_{i=1}^M {\lambda_i} \psi_i(x) \psi(y)$, with $M \in \Nbb \cup \{+\infty\}$, then $h(x) =  \sum_{i=1}^M {\lambda_i} \psi_i(x)^2$ and $h\mu$ can be sampled as a mixture of distributions $\psi_i^2 \mu$ with weights $\lambda_i$. An alternative is to sample independent Bernoulli variables $B_i \sim B(\lambda_i)$, and then sample  from the distribution $h_I \mu$ with  $h_I(x) =  \frac{1}{\#I}  \sum_{i\in I} \psi_i(x)^2$ with $I = \{i: B_i =1\}$. 
\end{remark}



\section{Least-squares with determinantal point processes and volume sampling}
\label{sec:dpp}


\subsection{Projection determinantal point process}
A projection determinantal point process $\DPP_\mu(V_m)$ associated with the space $V_m$ and the reference measure $\mu$ (\revb{not necessarily a finite measure}) is a distribution over $\Xc^m$ defined by 
$$
d\gamma_m(\x) = \frac{1}{m!} \det(\Phi(\x)^T \Phi(\x)) d\mu^{\otimes m}(\x), \quad \x\in \Xc^m,
$$
where $\Phi(\x) \in  \Rbb^{n \times m}$ is the matrix whose $i$-th row is $\varphib(x_i)^T$. It is a determinantal point process with projection kernel $K(x,y) = \varphib(x)^T \varphib(y)$ and reference measure $\mu$. 
Sampling from $\gamma_m$ tends to select  at random a set of features $(\varphib(x_1) , \hdots, \varphib(x_m))$ with high volume in $\Rbb^m$. The density  $ \frac{1}{m!}  \det(\Phi(\x)^T \Phi(\x))$  is equal to zero when  two points $x_i$ and $x_j$ \rev{are equal, or more generally when} two features $\varphib(x_i)$ and $\varphib(x_j)$ are \rev{collinear} for $i\neq j$. It is a particular class of repulsive point processes.  
The following result indicates that the marginals of $\gamma_m$  are all equal to the optimal sampling measure for i.i.d. sampling for $V_m$, and provides a factorization of the distribution in terms of conditional distributions. 
\begin{proposition}[Theorem 2.7 in \cite{lavancier2015dpp}]\label{prop:marginaldpp}
Let $ (x_1, \hdots , x_m) \sim \gamma_m$. Each $x_k$ has for marginal distribution 
$\nu_m = \revb{w_m^{-1}} \mu$, with $\revb{w_m^{-1}}(x) = \frac{1}{m} \Vert \varphib(x) \Vert_2^2$. For $2\le k \le m$, the conditional distribution of $x_{k}$ knowing $x_1,\hdots,x_{k-1}$ has for \revb{probability} density with respect to $\mu$ the function
$$
p_k(x_k)  := \frac{1}{m-k+1} \Vert \varphib(x_k) - P_{W_{k-1}} \varphib(x_k)  \Vert^2_2  
$$
where $P_{W_{k-1}}$ is the orthogonal projection onto the space $W_{k-1} = 
span \{\varphib(x_1) , \hdots , \varphib(x_{k-1})\} $ in $\Rbb^m.$
\end{proposition}
\begin{proof}
See \Cref{app:dpp}.
\end{proof}
From the previous result, we deduce a sequential procedure to draw a sample $(x_1 , \hdots,x_m)$ from the distribution $\gamma_m = \DPP_\mu(V_m)$. The first point $x_1$ is obtained by drawing a sample from $\nu_m.$ Then 
given the points $(x_1,\hdots,x_{k-1})$, the point $x_{k}$ is drawn 
from the \revb{probability} measure $p_k(x) d\mu(x)$.
\begin{example}
Consider $\Xc = [0,1]$ equipped with the uniform measure $\mu$ and the space $V_m$ of piecewise constant functions on a uniform partition of $[0,1]$ with $m$ intervals. An orthogonal basis is given by $\varphi_j(x) = \sqrt{m} \mathbf{1}_{x\in [(j-1)/m , j/m)}$. Here $ \varphib(x_i) = \sqrt{m} \e_i   $, where $\e_i$ is the $i$-th canonical vector \rev{in $\Rbb^m$}. Then the density of $\gamma_m$ is $0$ once two points or more are in the same interval, and equal to $m^m/m!$ if there is exactly one point in each interval. The marginals are all equal to $\mu$. The conditional density $p_k$ is equal to $0$ on the intervals containing the points $(x_1, \hdots,x_{k-1}),$ and equal to $\frac{m}{m-k+1}$ elsewhere. 
\end{example}


\begin{remark}
Letting $\v_1 , \hdots , \v_m$ be the orthonormal basis of $\Rbb^m$ such that $\v_i \propto \varphib(x_i) - P_{W_{i-1}} \varphib(x_i)$, we have that the functions $\psi_i(x) =  \v_i^T \varphib(x)$ form an $L^2_\mu$-orthonormal basis of $V_m$, and 
$$
p_k(x) d\mu(x)   =   \frac{1}{m-k +1} \left(\sum_{i=k}^m \psi_i(x)^2\right) d\mu(x),
$$
that is the optimal sampling distribution for the space $\mathrm{span}\{\psi_{k} , \hdots , \psi_m\}$, which is the orthogonal complement of $\mathrm{span}\{\psi_1, \hdots, \psi_{k-1}\}$ in $V_m.$
\end{remark}

\begin{remark}
When replacing the random draw
$
x_{k+1} \sim \frac{1}{m-k} \Vert \varphib(x) - P_{W_{k}} \varphib(x)  \Vert^2_2 d\mu(x)
$
by a deterministic selection 
$$
x_{k+1} \in \arg\max_{x \in \Xc}\Vert \varphib(x) - P_{W_{k}} \varphib(x)  \Vert^2_2,
$$
the resulting algorithm corresponds to a deterministic greedy algorithm for the construction of a hierarchical sequence of spaces 
$W_1\subset \hdots   \subset W_m$ for the approximation of the manifold $\mathcal{M} = \{\varphib(x) : x\in \Xc\}$ \cite{Maday2008Sep,binev2011convergence}. It also coincides with the sequential design in Gaussian process interpolation, using a kernel $K(x,y) = \varphib(x)^T\varphib(y)$. Indeed, in this case, $W_k = span\{ K(x_1,\cdot), \hdots, K(x_k,\cdot)\}$ and the interpolation of a function at points $(x_1,\hdots ,x_k)$ is the orthogonal projection onto $W_k$ with respect to the RKHS associated with the kernel $K$. The 
 variance at point $x$  of this interpolation given $(x_1, \hdots,x_k)$ is $\Vert \varphib(x) - P_{W_{k}} \varphib(x)  \Vert^2_2$. Therefore, the selected point $x_{k+1}$ is where the interpolation has maximum uncertainty.  
\end{remark}
For any \revb{probability density   $\revb{w^{-1}}$ w.r.t. $\mu$}, we let $\varphib^w : \Xc \to \Rbb^m $ be the weighted feature map such that $\varphib^w(x) = (\varphi^w_1(x),\hdots,\varphi^w_m(x))^T = \revb{w(x)^{1/2}} \varphib(x)$ and  $\Phi^w(\x)$ be the matrix in $ \Rbb^{n \times m}$ whose $i$-th row is $\varphib^w(x_i)^T$. We have the following straightforward property.   
\begin{proposition}\label{prop:dpp-equiv}
For any distribution $\nu = \revb{ w^{-1}} \mu$, it holds 
$$
d\gamma_m(\x)  =\frac{1}{m!} \det(\Phi^w(\x)^T \Phi^w(\x)) d\nu^{\otimes m}(\x), \quad \x\in \Xc^m.
$$
The functions $\varphi^w_1, \hdots, \varphi^w_m$ form an orthonormal basis of a subspace $V_m^w$ in $L^2_\nu(\Xc)$, 
and the distribution $\DPP_\mu(V_m) $ coincides with $\DPP_\nu(V_m^w).$ 
\end{proposition}
From the above, we deduce  that for $\nu =  \revb{ w^{-1}} \mu$,
$$
d\gamma_m(\x) = \frac{m^m }{m!} \det(\G^{w}(\x)) d\nu^{\otimes m}(\x), \quad \x\in \Xc^m. 
$$
Therefore, 
sampling from $\gamma_m$ tends to favor points $\x \in \Xc^m$ leading simultaneously  to a high likelihood with respect to the  product measure $\nu^{\otimes m}$ and a high value of the determinant of $\G^{w}(\x)$. This tends to favor empirical Gram matrices with high eigenvalues.  

%

\revb{
\begin{remark}[Complexity]\label{rem:dpp-sampling-complexity}
Let us assume that $\mu$ is a discrete measure with $N$ atoms. The cost of  sampling  a projection DPP is  $O(m^3 + N m^2)$. This can be improved to $O(m^3 + N m)$ (up to log factors) using rejection sampling, see \cite{Barthelme2023Apr}. 
\end{remark}
}

\subsection{Volume sampling }


The volume sampling distribution $ \VS^n_\mu(V_m )$ is the distribution over $\Xc^n$ defined by
$$
d\gamma_n(\x) = \frac{(n-m)!}{n!} \det(\Phi(\x)^T \Phi(\x)) d\mu^{\otimes n}(\x),
$$
for $n\ge m.$ For $n=m$, the volume sampling distribution $\VS^m_\mu(V_m) $ coincides with the projection determinantal point process $ \DPP_\mu(V_m)$.
For $n>m,$ \revb{provided $\mu $ is a probability measure}, a sample from $\gamma_n$ is  composed by 
  $m$ samples from the projection determinantal process $\DPP_\mu(V_m)$ and $n-m$ i.i.d. samples from the measure $\mu$, to which is applied a random permutation, as stated in the next proposition.
\begin{theorem}[Theorem 2.4 in \cite{derezinski2022unbiased}]\label{th:pvs_samples}
\revb{Assume that $\mu$ is a probability measure.} If $(x_1,\hdots,x_n) \sim \gamma_m \otimes \mu^{\otimes (n-m)}$ and  
$\sigma$ is an independent permutation drawn uniformly at random over the set of permutations of $\{1,\hdots,n\}$, then $(x_{\sigma(1)},\hdots,x_{\sigma(n)}) \sim \gamma_n.$ 
The marginals of the distribution $\gamma_n$ are all equal to the mixture 
$$
\frac{m}{n} \nu_m + \frac{n-m}{n} \mu = (\frac{m}{n} w_m + \frac{n-m}{n}) \mu.
$$
\end{theorem}

Given a \revb{probability} measure $  \nu =   \revb{w^{-1}} \mu$ \revb{($\mu $ is no more required to be a probability measure)},  for $n\ge m$ we can define another volume sampling distribution 
$\VS_{  \nu}^n(V_m^{  w} )$  over $\Xc^n$  defined by
$$
d\gamma^{  \nu}_n(\x) =  \frac{(n-m)!}{n!} \det(\Phi^{  w}(\x)^T \Phi^{  w}(\x)) d{ \nu}^{\otimes n}(\x) = n^m \frac{(n-m)!}{n!}  \det(\G^{  w}(\x)) d{  \nu}^{\otimes n}(\x).
$$
Sampling from $\gamma_n^\nu$ tends to favor points $\x \in \Xc^m$ leading simultaneously  to a high likelihood with respect to the  product measure $\nu^{\otimes n}$ and a high value of the determinant of $\G^{w}(\x)$. As a corollary of Theorem \ref{th:pvs_samples}, we have the following result.

\begin{theorem}\label{th:vs-marginals}
If $(x_1,\hdots,x_n) \sim \gamma_m \otimes \nu^{\otimes (n-m)}$, with $\nu = \revb{ w^{-1}} \mu$ \revb{a probability measure},  and  
$\sigma$ is an independent permutation drawn uniformly at random over the set of permutations of $\{1,\hdots,n\}$, then $(x_{\sigma(1)},\hdots,x_{\sigma(n)}) \sim \gamma^{  \nu}_n.$
The marginals of the distribution $\gamma_n^{  \nu}$ are all equal to the mixture 
$$
\tilde \nu = \revb{ \tilde w^{-1}}  \mu \quad \text{with} \quad \revb{ \tilde w^{-1}}  =  \frac{m}{n} \revb{ w_m^{-1}} + \frac{n-m}{n} \revb{ w^{-1}}.
$$
If $\revb{w_m\ge \alpha w}$, then $\tilde w$ satisfies
$$
(1 - \frac{m}{n} )  \revb{ w^{-1}} \le  \revb{\tilde w^{-1}}
\le (1 + (\alpha^{-1}-1) \frac{m}{n})  \revb{  w^{-1}}.
$$ 
\end{theorem}
\begin{proof}
Since $\varphib^w$ form an orthonormal basis of the space $V_m^w$ in  $L^2_\nu$, we deduce from Theorem \ref{th:pvs_samples} and Proposition \ref{prop:dpp-equiv}
that up to a random permutation, a sample  from $\gamma_n^{  \nu}$  is composed by $m$ points drawn from $\DPP(V_m^{  w} , {  \nu}) = \DPP(V_m , \mu)$ (with marginals $\nu_m$)  and   $n-m$ i.i.d. samples from the measure ${  \nu}$. The expression of the marginals is a direct consequence. 
\end{proof}

 Taking ${  \nu} = \mu$ \revb{(provided $\mu$ is a probability measure)}, we have $\gamma^\mu_n = \gamma_n$, that is the classical volume sampling distribution $\VS_\mu^n(V_m)$.  
 Taking ${  \nu} = \nu_m$, we obtain the distribution
$$
d\gamma^{\nu_m}_n(\x) =  \frac{(n-m)!}{n!} \det(\Phi^{w_m}(\x)^T \Phi^{w_m}(\x)) d\nu_m^{\otimes n}(\x)
$$ 
which corresponds to the {volume-rescaled sampling} distribution from \cite[Section 3]{derezinski2022unbiased}, whose marginals are all equal to the optimal sampling measure $\nu_m$ (leverage score sampling). Up to a random permutation, this consists of $m$  samples  from $\gamma_m$ and $n-m$ i.i.d. samples from the optimal sampling measure $\nu_m$. Considering $\gamma^{  \nu}_n$ with ${  \nu} \neq \nu_m$ will further allow us to obtain $H \to L^2_\mu$ quasi-optimality result in probability.

\subsection{Properties of least-squares projection}

In this section, we consider weighted least-squares projection based on volume sampling with reference probability measure $\nu = \revb{w^{-1}} \mu$. The case $\nu= \nu_m$ corresponds to volume-rescaled sampling and enjoy favorable properties for the error in expectation.
However, as we will see, taking $\nu$ as a mixture allows us to obtain a control of errors with high probability. 

We first state some results on the minimal  eigenvalue of the Gram matrix when using volume sampling distribution $\gamma_n^\nu$.   This is a straightforward extension of Theorem 2.9 from  \cite{derezinski2022unbiased}.
\begin{lemma}\label{lem:vs-invG}
Assume $\x$ is drawn from the distribution $\gamma_n^{  \nu}$ with $\nu= \revb{w^{-1}}\mu$ \revb{a probability measure}. It holds 
$$
\Ebb((\G^{  w})^{-1}) \preceq \frac{n}{n-m+1} \I,
$$
where the Loewner ordering $\preceq$ is replaced by an equality whenever  the matrix $\Phi^w(\y)$ for $\y\sim \nu^{\otimes n}$ has rank $m$ almost surely, and 
$$
\Ebb(\lambda_{min}(\G^{  w})^{-1}) \le \frac{nm}{n-m+1}.
$$
\end{lemma}
\begin{proof}
See \Cref{app:vs}.
\end{proof}
Provided a condition on a minimal number of samples, the next result improves the above upper bound by exploiting 
a matrix concentration inequality. 
\begin{lemma}\label{lem:vs-lambdamin}
Assume $\x$ is drawn from the distribution $\gamma_n^{  \nu}$ with $  \nu = \revb{w^{-1}}\mu$ and $\revb{w^{-1}}\ge \alpha \revb{w_m^{-1}}$.
Then  
$$
\Pbb(\lambda_{min}(\G^w)^{-1}> (1-\delta)^{-1} \frac{n}{n-m})  \le m \exp(-\frac{c_\delta (n-m) \alpha}{m}).
$$
Moreover, if   
$$
n \ge m + mc_\delta^{-1} \alpha^{-1} \log(nm^2)
$$
 it holds 
$$
\Ebb(\lambda_{min}(\G^w)^{-1}) \le 1 + \frac{n}{n-m} (1-\delta)^{-1}.
$$
\end{lemma}
\begin{proof}
See \Cref{app:vs}.
\end{proof}

\begin{proposition}\label{prop:vs-variance_projection}
Let $\x = (x_1,\hdots,x_n)$ be drawn from the distribution $\gamma_n^\nu$ with $\nu = \revb{w^{-1}}\mu$ and $\revb{w^{-1}}\ge \alpha \revb{w_m^{-1}}$. Assume we use weighted least-squares with weight function $w$. Then for any function $f$, letting $S_t = \{\lambda_{min}(\G^w(\x))^{-1} \le  t \}$, it holds
$$
\Ebb(\Vert \hat P_{V_m}  f\Vert^2 \vert S_t) \le \Pbb(S_t)^{-1} t  (1 - \frac{m}{n})^{-1} \Vert f \Vert^2,
$$  
and 
\begin{align*}
  \Ebb(\Vert \hat P_{V_m} f \Vert^2 | S_t) &\le \Pbb(S_t)^{-1} t^{2} (  \frac{m}{n} \alpha^{-1} ( \beta + \xi m \alpha^{-1})   \Vert f \Vert^2  +   \Vert P_{V_m} f \Vert^2),
\end{align*}
with $\beta =1 + (\alpha^{-1}-1) \frac{m}{n}$,  
 $\xi   = 0$ if $\nu = \nu_m$, $\xi=1$ in the case $\nu\neq \nu_m$. 
If $n \ge m + c_\delta^{-1} \alpha^{-1} m\log(m\eta^{-1})$ and $t = (1-\delta)^{-1} \frac{n}{n-m}$,  then $\Pbb(S_t)    \ge 1-\eta$. 
\end{proposition}
\begin{proof}
For the first inequality, we note that $$\Ebb(\Vert \hat P_{V_m} f \Vert^2 \vert S_t) \le t \Pbb(S_t)^{-1} \Ebb( \Vert \hat P_{V_m} f \Vert_n^2 )  \le  t \Pbb(S_t)^{-1} \Ebb( \Vert f \Vert_n^2 ),$$
where we have used the fact that $\hat P_{V_m}$ is an orthogonal projection with respect to $\Vert \cdot \Vert_n.$ Then since the marginals of $\x$ are $\revb{\tilde w^{-1}}   \mu$ with $\revb{\tilde w^{-1}} \ge (1 - \frac{m}{n} ) \revb{w^{-1}}$ (Theorem \ref{th:vs-marginals}), we deduce from Lemma \ref{lem:semi-norm-bound} that  
$\Ebb( \Vert f \Vert_n^2 ) \le (1 - \frac{m}{n})^{-1} \Vert f\Vert^2.$ 
For the second inequality, 
we note that $\Vert \hat P_{V_m} f \Vert = \Vert \c \Vert_2^2$ with $\c = \G^w(\x)^{-1} \b$ and $\b = \frac{1}{n} \Phi^w(\x)^T f^w(\x),$ where $f^w = f \revb{w^{1/2}}.$ 
Then noting that $\Vert \c \Vert_2 \le \Vert \G^w(\x)^{-1} \Vert_2 \Vert \b \Vert_2 = \lambda_{min}(\G^w(\x))^{-1} \Vert_2 \Vert \b \Vert_2 , $ we have 
\begin{align*}
 \Ebb(\Vert \c \Vert_2^2 | S_t) &\le  t^{2} \Ebb(\Vert \b \Vert_2^2 | S_t) \le \Pbb(S_t)^{-1} t^{2} \Ebb(\Vert \b \Vert_2^2),
\end{align*}
and the result follows from Lemma \ref{lem:vs-variance_quasi_projection} and Lemma \ref{lem:vs-lambdamin}.
\end{proof}

\begin{theorem}\label{th:vs-quasi-optimality-expectation}
Assume $\x$ is drawn from the distribution $\gamma_n^{  \nu}$ with $  \nu = \revb{ w^{-1}} \mu$ and  $\revb{ w^{-1}} \ge  \alpha \revb{ w_m^{-1}} $, and assume we use weighted least-squares with weight function $w$. 
If $$n \ge m + c_\delta^{-1} \alpha^{-1} m\log(m\eta^{-1}),$$
with $0< \delta<1,$ then the event $S = \{\lambda_{min}(\G^w) \ge (1-\delta) \frac{n-m}{n})\}$ is such that $\Pbb(S) \ge 1 - \eta$, and it holds 
$$
\Ebb(\Vert f - \hat f_m \Vert^2 \vert S) \le (1 +  (1-\eta)^{-1} (1-\delta)^{-1} (1-\frac{m}{n})^{-1} \beta  ) \inf_{g\in V_m} \Vert f - g \Vert^2,
$$
 and  
$$
\Ebb(\Vert f - \hat f_m \Vert^2 \vert S) \le (1 +  (1-\eta)^{-1} (1-\delta)^{-2} (1-\frac{m}{n})^{-2}  \frac{m}{n} \alpha^{-1} ( \beta + \xi m \alpha^{-1})) \inf_{g\in V_m} \Vert f - g \Vert^2,
$$
with  $\beta =1 + (\alpha^{-1}-1) \frac{m}{n}$ and $\xi = 1$ if $\nu\neq \nu_m$ or $\xi=0$ if $\nu=\nu_m$.
\end{theorem}
\begin{proof}
 Lemma \ref{lem:vs-lambdamin} implies $\Pbb(S) \ge 1-\eta.$  
The marginal distributions are all equal to $\revb{ \tilde w^{-1}} \mu$ with $\revb{ \tilde w^{-1}} \le \beta \revb{   w^{-1}}$. Then  
the first inequality follows from Lemma \ref{lem:quasi-optimality-conditional-expectation}, and the second inequality  follows from  Proposition \ref{prop:vs-variance_projection}.
\end{proof}
We next provide a result in probability and another result in expectation (without conditioning) under the assumption that the target function $f$ is in some subspace $H$ of $L^2_\mu$. 
\begin{theorem}\label{th:vs-quasi-optimality-H}
Assume that $f\in H$, with $H$ satisfying \eqref{H-embedding} and \eqref{H-hcondition}, with $h$ a \revb{probability} density with respect to $\mu$. 
Assume that $(x_1,\hdots,x_n)$ is drawn from $\gamma_n^\nu$ with $\nu = \revb{   w^{-1}} \mu$ and $\revb{   w^{-1}} =   \alpha \revb{   w_m^{-1}} +  ({1-\alpha}) h$, and  we use weighted least-squares with weight function $w$.
  Then it holds 
\begin{align*}
\Vert f - \hat f_m \Vert &\le \left( C_H+ (1-\delta)^{-1/2} (1-\alpha)^{-1/2}  {({1-\frac{m}{n} })^{-1/2} }\right) \inf_{g\in V_m}\Vert f - g \Vert_H  
\end{align*}
with probability greater than $1-m \exp(-\frac{c_\delta (n-m) \alpha}{m}),$
and if 
$$
n \ge m + c_\delta^{-1} \alpha^{-1} m \log(nm^2)
$$
 it holds 
\begin{align*}
\Ebb(\Vert f - \hat f_m \Vert^2) &\le \left ( 2C_H^2+2 (1-\alpha)^{-1}  ( 1+  (1-\delta)^{-1} \right (  {{1-\frac{m}{n} })^{-1}} ) \inf_{g\in V_m}\Vert f - g \Vert_H ^2
\end{align*}
\end{theorem}
\begin{proof}
It is directly deduced from Lemma \ref{lem:vs-lambdamin} with $\zeta = (1-\alpha)^{-1}$ and Lemma \ref{lem:near-optimality-H}. 
\end{proof}
Note that the  result in expectation from Theorem \ref{th:vs-quasi-optimality-H} does not require to use conditioning for ensuring the stability of the Gram matrix.  

\reva{\begin{remark}
Quasi-optimality guarentees from Theorems \ref{th:vs-quasi-optimality-expectation} and \ref{th:vs-quasi-optimality-H} are obtained under the condition $n \gtrsim m\log(m) $ on the sampling complexity, which is similar to the results from Theorems \ref{th:quasi-optimality-iid-expectation} and \ref{th:quasi-optimality-iid-H}, respectively for i.i.d. sampling. The numerical experiments will confirm that the requirement on the number of samples required to obtain stability is similar for volume-rescaled and i.i.d. sampling. However,  in terms of approximation errors, we observe in practice that volume-rescaled sampling outperforms i.i.d. sampling. 
\end{remark}
}

\paragraph{Unbiased projection and aggregation of projections.}

We next state a remarkable result, proven in \cite[
Theorem 3.1]{derezinski2022unbiased} for classical and volume-rescaled sampling, showing that with such sampling, the projection $\hat f_m = \hat P_{V_m} f$  is an unbiased estimation of the element of best approximation $f_m = P_{V_m} f$. The result is here stated for the distribution $\gamma_n^\nu$ with a general probability measure $\nu.$
\begin{theorem}\label{th:vs-unbiased}
Assume $(x_1,\hdots,x_n)$ is drawn from the distribution $\gamma_n^\nu$ with \revb{probability measure} $\nu = \revb{   w^{-1}} \mu$ and we use weighted least-squares with weight function $w$. Then for any $f\in L^2_\mu$, it holds
$$
\Ebb(\hat P_{V_m} f) = P_{V_m}f.
$$
\end{theorem}
\begin{proof}
We have $\hat P_{V_m} f(\cdot) = \varphib(\cdot)^T \Phi^w(\x)^\dagger f^w(\x)$ with   $f^w = f\revb{w^{1/2}}.$ Then using  Lemma \ref{lem:vs-unbiased-pseudoinverse}, we obtain
$\Ebb(\hat P_{V_m} f(\cdot)) = \varphib(\cdot)^T \Ebb( \Phi^w(\x)^\dagger f^w(\x)) = \varphib(\cdot)^T  \revb{\int \varphib(y) f(y) d\mu(y)}  = P_{V_m} f$.
\end{proof}
The next result shows a stability of empirical projection in expectation, and hence a quasi-optimality in expectation, which does not require a conditioning to ensure stability of the Gram matrix. It extends  \cite[
Theorem 3.1]{derezinski2022unbiased}  to volume sampling with general reference measure $\nu$.
\begin{theorem}\label{th:vs-quasioptim}
Assume $(x_1,\hdots,x_n)$ is drawn from the distribution $\gamma_n^\nu$ with $\nu=\revb{   w^{-1}}\mu$ such that $\revb{   w^{-1}} \ge \alpha \revb{   w_m^{-1}}$ and we use weighted least-squares with weight function $w$. 
Provided   $n\ge 2m+2$ and $n \ge 2 m \alpha^{-1} c_{\delta}^{-1} \log(\zeta^{-1} m^2 n )$, it holds 
$$
\Ebb(\Vert \hat P_{V_m} g\Vert^2)  \le (4\frac{m}{n}(1-\delta)^{-2} (\beta + \xi m \alpha^{-1}) + \alpha^{-1} \zeta ) \Vert g \Vert^2  + 4(1-\delta)^{-2} \Vert P_{V_m} g \Vert^2
$$
for any $g\in L^2_\mu$, where $\xi=0$ for $\nu=\nu_m$ or $\xi = 1$ for $\nu\neq \nu_m$, and $\beta = 1+(\alpha^{-1} -1) \frac{m}{n}$.   
Then 
provided  $$n \ge C( m\log(\epsilon^{-1}m) + m \epsilon^{-1})$$
 for a sufficiently large $C$, 
 it holds  
\begin{equation}
\Ebb(\Vert f - \hat P_{V_m} f \Vert^2) \le (1 + \epsilon (1+\xi m) ) \Vert f - P_{V_m} f \Vert^2\label{vs-quasi-optim-expectation-without-conditioning}
\end{equation}
with $\xi = 0$ for $\nu=\nu_m$ or $\xi = 1$ for $\nu\neq \nu_m.$
 \end{theorem}
\begin{proof}
We have 
$$
\Ebb(\Vert f - \hat P_{V_m} f \Vert^2) = \Vert f - P_{V_m} \Vert^2 + \Ebb(\Vert \hat P_{V_m} (f - P_{V_m} f)\Vert^2).$$
Let $g = f - P_{V_m} f$. Note that $\Ebb(\Vert \hat P_{V_m} g\Vert^2) = \Ebb(\Vert \Phi^w(\x)^\dagger g^w(\x) \Vert_2^2)$. Then using Lemma \ref{lem:vs-variance_projection}, we show that provided  $n\ge 2m+2$ and $n \ge 2 m \alpha^{-1} c_{\delta}^{-1} \log(\zeta^{-1} m^2 n )$, it holds 
$$
\Ebb(\Vert \hat P_{V_m} g\Vert^2)  \le (4\frac{m}{n}(1-\delta)^{-2}(\beta + \xi m \alpha^{-1}) + \alpha^{-1} \zeta ) \Vert g \Vert^2  
$$
with $\beta = 1 + (\alpha^{-1} -1) \frac{m}{n} $, and $\xi = 0$ if $\nu=\nu_m$ or $\xi = 1$ if $\nu\neq \nu_m.$ 
The condition $n \ge 2 m \alpha^{-1} \delta^{-2} \log(\zeta^{-1} m^2 n )$ can be converted into $n \ge C' m \log(\zeta^{-1} m)$ for some $C'.$
Therefore, provided $n \ge C( m\log(\epsilon^{-1} m) + m\epsilon^{-1}) $ with a sufficiently large $C$, it holds 
$$\Ebb(\Vert f - \hat P_{V_m} f \Vert^2) \le (1 + \epsilon (1+\xi m)) \Vert f - P_{V_m} f \Vert^2$$
with $\xi = 0$ for $\nu=\nu_m$ or $\xi = 1$ for $\nu\neq \nu_m.$

\end{proof}
The above results allow to analyze the property of an aggregation of $r$ independent least-squares projections based on volume sampling, that yields a quasi-optimality result in expectation (without conditioning), and a convergence to best approximation when $r\to \infty. $

 \begin{corollary}
Let $r\in \Nbb$. Let $\hat f^{(1)}, \hdots, \hat f^{(r)}$ be $r$ independent least-squares projections constructed from independent samples $\x^{(1)}, \hdots, \x^{(r)}$ drawn from $\gamma^\nu_n$, with $\nu = \revb{   w^{-1}} \mu$ such that $\revb{   w^{-1}} \ge \alpha \revb{   w_m^{-1}}$,  and using weighted least-squares with weight $w$. Then provided  $n \ge C( m\log(\epsilon^{-1} m) + m(1+\xi m)\epsilon^{-1}) $ with sufficiently large $C$, the averaged estimator $\bar  f^r = \frac{1}{r} \sum_{k=1}^r \hat f^{(k)}$ satisfies 
$$
\Ebb(\Vert f - \bar  f^r \Vert^2) \le (1+ \frac{1}{r} (1 + \xi m) \epsilon) \Vert f -  P_{V_m} f\Vert^2,
$$
with $\xi = 0$ for $\nu=\nu_m$ or $\xi = 1$ for $\nu\neq \nu_m.$
\end{corollary}
\begin{proof}
The estimators $\hat f^{(k)}$ are independent and follow the distribution of an estimator $ \hat P_{V_m} f$ constructed with samples drawn from $\gamma_n^\nu.$ 
From Theorem \ref{th:vs-unbiased}, we have that $\Ebb(\hat f^{(k)}) = P_{V_m}  f $ for all $k$. 
Then using the independence of the $\hat f^{(k)}$ and Theorem \ref{th:vs-quasioptim}, we obtain
\begin{align*}
\Ebb(\Vert f - \bar f^r \Vert^2)& =  \Vert f -  P_{V_m} f\Vert^2 + \Ebb(\Vert  P_{V_m}  f - \bar f^r \Vert^2)\\
&\le \Vert f -  P_{V_m} f\Vert^2 + \frac{1}{r}  \Ebb(\Vert  P_{V_m}  f - \hat P_{V_m} f\Vert^2)
\\&\le (1+ \frac{1}{r} (1 + \xi m) \epsilon) \Vert f -  P_{V_m} f\Vert^2.
\end{align*} 
provided  $n \ge C( m\log(\epsilon^{-1} m) + m \epsilon^{-1}) $ with sufficiently large $C$.
\end{proof}
The quasi-optimality constant  $(1+ \frac{1}{r} (1 + \xi m) \epsilon)$ is optimal when $\xi=0$, i.e. $\nu = \nu_m$. When $\nu\neq \nu_m$, having quasi-optimality requires either $\epsilon \sim m^{-1}$ for fixed $r$, or $r\sim m$ for fixed $\epsilon$, both cases yielding a condition on the total number of samples in $nr \sim m^2$, which is suboptimal compared to the case $\nu = \nu_m$ and even compared with i.i.d. sampling. However, no conditioning is required. Also,  
there is  an interest in using the volume sampling distribution $\gamma_n^\nu$ with $\nu \neq \nu_m$
 in order to obtain simultaneously a guaranty in expectation (yet suboptimal) and a guaranty in probability for functions from a specific function space $H$.  Indeed, as a corollary of Theorem \ref{th:vs-quasi-optimality-H}, and under the assumptions of these theorems, we obtain using a simple union bound that
 \begin{align*}
\Vert f - \bar f^r \Vert &\le \left( C_H+ (1-\delta)^{-1/2} (1-\alpha)^{-1/2} (1-\frac{m}{n})^{-1}  \right) \inf_{g\in V_m}\Vert f - g \Vert_H  
\end{align*}
with probability greater than $1-r m \exp(-\frac{c_\delta (n-m) \alpha}{m}).$ 
%

\section{Least-squares with independent repetitions of volume sampling} 
\label{sec:repeated}

We now consider approximation methods relying on independent repetitions from the volume  sampling distribution. We will first consider repetitions of projection DPP distribution $\gamma_m$ and prove some results in expectation. Next we will consider repetitions of general volume sampling distribution  $\gamma_n^\nu$, which allows to obtain  results both in expectation and in probability for some specific function spaces, with a suitable choice of the measure $\nu$.  

\subsection{Independent repetitions of DPP distribution $\gamma_m$}
We consider $\x = (\x_1 , \hdots,\x_r)$ where 
the $\x_k = (x_{1,k} , \hdots , x_{m,k})  $ are i.i.d. samples from $\gamma_{m}$, and the corresponding weighted least-squares projection using $n= mr$ points. We consider least-squares with the optimal weight function $w_m$.
The empirical Gram matrix can be written 
$$
\G^{w_m}(\x) = \frac{1}{r} \sum_{k=1}^r \G^{w_m}(\x_k), 
$$
with 
$$  \G^{w_m}(\x_k) = \frac{1}{m} \sum_{i=1}^m \varphib^{w_m}(x_{i,k}) \varphib^{w_m}(x_{i,k})^T =  \sum_{i=1}^m \frac{ \varphib(x_{i,k}) \varphib(x_{i,k})^T}{\Vert \varphib(x_{i,k})\Vert_2^2} 
$$
where the $\G^{w_m}(\x_k)$ are i.i.d. matrices with expectation $\I$ and spectral norm bounded by $ m$. Using conditioning, we have the following result.
\begin{theorem}\label{th:repeated-dpp-quasi-optimality-expectation}
Assume $\x$ is drawn from the distribution $\gamma_m^{ \otimes r}$, and assume we use weighted least-squares with weight function $w_m$. Letting $S_\delta = \{\lambda_{min}(\G^{w_m}(\x) ) \ge 1-\delta\}$, it holds 
$$
\Ebb(\Vert f - \hat f_m \Vert^2 \vert S_\delta) \le (1 +  \Pbb(S_\delta)^{-1} (1-\delta)^{-1} ) \inf_{g\in V_m} \Vert f - g \Vert^2,
$$
and  
$$
\Ebb(\Vert f - \hat f_m \Vert^2 \vert S_\delta) \le (1 +  \Pbb(S_\delta)^{-1} (1-\delta)^{-2} r^{-1} )\inf_{g\in V_m}  \Vert f - g \Vert^2.
$$
\end{theorem}
\begin{proof}
This is a particular case of Theorem \ref{th:repeated-vs-quasi-optimality-expectation} with $\bar n = m$ and $\nu= \nu_m$, where 
$\gamma_m^\nu = \gamma_m$, $\alpha = 1$ and $\xi=0$.
\end{proof}

It remains to control the probability of the event $S_\delta = \{\lambda_{min}(\G^{w_m}(\x) ) \ge 1-\delta\}.$
From Matrix Chernoff inequality (Lemma \ref{lem:matrix-chernoff}), we deduce that 
$$
  \Pbb( \lambda_{min}(\G^{w_m}(\x) ) < 1-\delta ) \le m \exp(- \frac{r c_\delta }{m})  .
$$
and we conclude that  
$$
  \Pbb( \lambda_{min}(\G^{w_m}(\x) ) < 1-\delta ) \le \eta
$$
provided $n = rm \ge c_\delta^{-1} m^2  \log( m \eta^{-1})$. This result is suboptimal compared to i.i.d. sampling from $\nu_m$, but it does not exploit the properties of DPP, which may yield to matrices  $\G^{w_m}(\x_k)$ with spectral norm (much) lower than $m$ with high probability. We have to better analyse the distribution of the random matrix 
$$
\A(\x) := \G^{w_m}(\x) = 
\sum_{i=1}^m \frac{ \varphib(x_{i}) \varphib(x_{i})^T}{\Vert \varphib(x_{i})\Vert_2^2} , \quad \x= (x_1,\hdots , x_m)\sim \gamma_m . 
$$
In particular, if the distribution $\gamma_m$ is such that the $\varphib(x_1), \hdots,   \varphib(x_m)$ are close to orthogonal with high probability, then with high probability, $\A(\x)$ is close to identity and the least-squares problem is well conditioned. 
\begin{example}
An ideal situation occurs  when $V_m$ is the space of piecewise constant functions on a uniform  partition of $[0,1]$ with $m$ intervals, where $\varphi_j(x) = \sqrt{m} \mathbf{1}_{x\in [(j-1)/m , j/m)}$, $w_m = 1$, and $ \varphib(x_i) = \sqrt{m} \e_i   $, where $\e_i$ is the $i$-th canonical vector \rev{in $\mathbb{R}^m$}. Here the vectors $ \varphib(x_1), \hdots,   \varphib(x_m)$ are orthogonal almost surely, and $\A(\x)  = \I$ almost surely. 
\end{example}
  
Recall that we have $$\gamma_m(\x) = \frac{1}{m!} \det(\Phi(\x)^T\Phi(\x)) \mu^{\otimes m} = \frac{m^m}{m!} \det(\A(\x)) \nu_m^{\otimes m}(\x)$$
so that with $\x\sim \gamma_m$ and $\y  \sim \nu_m^{\otimes m}$, it is more likely to have matrices $\A(\x)$ with higher determinant than $\A(\y)$, and hence higher eigenvalues.
This leads us to make the following conjecture.
\begin{conjecture}\label{conj:dpp}
The distribution $\gamma_m$ satisfies 
\begin{align}
\mathbb{P}_{\z \sim \gamma_m}( F(  \A(\z)) > t  ) \le \mathbb{P}_{\y \sim \nu_m^{\otimes m}}( F(  \A(\y)) > t  )  \label{conj_2}
\end{align}
for $t>0$ and $F$ a real-valued positive, convex and   decreasing function in the Loewner order.
\end{conjecture}
If the distribution $\gamma_m$ satisfies the property of Conjecture \ref{conj:dpp},  we obtain the following result. 
\begin{proposition}\label{prop:repeated-dpp-proba-conj}
Let $\x= (\x_1,\hdots,\x_r) \sim \gamma_m^{\otimes r}$ with    $\gamma_m$ a distribution over $\Xc^m$ satisfying \eqref{conj_2}.
 Then 
$$\Pbb(\lambda_{min}(\G^{w_m}(\x)) < 1-\delta) \le m \exp(- \frac{c_\delta n }{m}).$$ 
\end{proposition}
\begin{proof}
We have 
$\Pbb(\lambda_{min}(\G^{w_m}(\x)) < 1-\delta) =   \Pbb(\lambda_{min}(\G^{w_m}(\x))^{-1} > t)$ 
 with $t:= (1-\delta)^{-1}$. The function $F := \B \mapsto \lambda_{min}(\B)^{-1}$ is a positive convex and monotonically decreasing function in the Loewner order. For any fixed symmetric positive semi-definite matrix $\H$, $\A \mapsto \lambda_{min}(\H + \A/r)^{-1}$ is also a  positive convex and monotonically decreasing function in the Loewner order. Letting 
 $\G^{w_m}(\x) = \frac{1}{r} \sum_{i=1}^m \A(\x_i) := \H(\x_1, \hdots, x_{r-1}) + \A(\x_r)/r$, we then have
   \begin{align*}
  \Pbb(F(\G^{w_m}(\x)) > t)  &= \Ebb(\mathbf{1}_{F( \frac{1}{r} \sum_{i=1}^r \A(\x_i)) > t}) \\
  &=  \Ebb(\Ebb(\mathbf{1}_{F( \H(\x_1, \hdots, \x_{r-1}) + \A(\x_r)/r)>t} \vert \x_1, \hdots, \x_{r-1}) ) \\
& \le \Ebb(\Ebb(\mathbf{1}_{F( \H(\x_1, \hdots, \x_{r-1}) + \A(\y_r)/r)>t} \vert \x_1, \hdots, \x_{r-1}) ) \\
& = \Ebb(\mathbf{1}_{F( \H(\x_1, \hdots, \x_{r-1}) + \A(\y_r)/r)>t} ),
\end{align*}
where $\y_r \sim \nu_m^{\otimes m}$. Letting $\y_1, \hdots,  \y_r$ be i.i.d. samples from  $\nu_m^{\otimes m}$, and 
successively conditioning on $(\x_1,\hdots,\x_{r-2}, \y_r)$, $(\x_1,\hdots,\x_{r-3}, \y_{r-1} ,  \y_r)$, ..., $(\y_2 \hdots,\y_r)$, we  obtain  
 \begin{align*}
 \Pbb(F(\G^{w_m}(\x)) > t) &\le  
  \Ebb(\mathbf{1}_{F(\frac{1}{r} \sum_{i=1}^r \A(\y_i))>t} ) = \Pbb(F(\G^{w_m}(\y)) > t),
\end{align*}
where $\y = (\y_1 , \hdots, \y_r) \sim \nu_m^{\otimes n}$, $n=rm$. Then 
it holds 
 $ \Pbb(\lambda_{min}(\G^{w_m}(\x)) < 1-\delta)  \le  \Pbb(\lambda_{min}(\G^{w_m}(\y)) < 1-\delta)  $,  and we conclude using Lemma \ref{lem:matrix-chernoff}.
\end{proof}
\reva{
The above result ensures that stability is controlled in probability with a number of samples which is at most the number of samples required by i.i.d. sampling from the optimal distribution $\nu_m$. In numerical experiments, we   observe that this number of samples is in fact much lower than with i.i.d. sampling. To be understood, this would require other tools for analyzing the concentration of $\G^{w_m}$.
}

\begin{remark}\label{rem:ass-lmin}
The proof of Proposition \ref{prop:repeated-dpp-proba-conj} exploits 
the assumption \eqref{conj_2} to prove that 
\begin{align}
\Pbb_{\z \sim \gamma_m}(\lambda_{min}(\H + \A(\z)/r)^{-1}>t) \le 
\Pbb_{\y \sim \nu_m^{\otimes m}}(\lambda_{min}(\H + \A(\y)/r)^{-1}>t)\label{ass_dpp_bis}
\end{align}
for any fixed p.s.d.  matrix $\H$. The assumption  \eqref{ass_dpp_bis} on $\gamma_m$ would be sufficient to obtain the result of Proposition \ref{prop:repeated-dpp-proba-conj}. 
\end{remark}

\begin{remark}\label{rem:ass-trace-exp}
In order to obtain the result of  Proposition \ref{prop:repeated-dpp-proba-conj}, an alternative assumption on $\gamma_m$ would be that 
\begin{align}
\Ebb_{\z \sim \gamma_m}( G(e^{s \A(\z)})) \le  \Ebb_{\y \sim \nu_m^{\otimes m}}( G(e^{s \A(\y)})) \label{ass_dpp_ter}
\end{align}
for any $s<0$ and $G$ a real-valued positive, concave and monotonically increasing function in the Loewner order. Under this assumption, we have to follow the proof of matrix Chernoff. The first steps of the proof of matrix Chernoff inequality (Theorem \ref{th:matrix-chernoff}) yield 
$$
\Pbb(\lambda_{min}(\frac{1}{r} \sum_{i=1}^r \A(\x_i)) < t ) \le \inf_{\theta <0} e^{-\theta t} \Ebb(\trace  \exp( \sum_{i=1}^r \theta  \A(\x_i)/r)).
$$
Letting  $\frac{1}{r}  \sum_{i=1}^r \A(\x_i) := \H + \rev{\theta}\A(\x_r)/r$, we have 
$$
\trace  \exp( \sum_{i=1}^r \theta  \A(\x_i)/r) = G( e^{\A(\x_r) \theta/r})$$ with $G := \X \mapsto \trace  \exp( \H + \log(\X))$ a concave and increasing function in the Loewner order. The assumption then implies that 
 $$\Ebb(\trace  \exp( \sum_{i=1}^r \theta  \A(\x_i)/r)) \le \Ebb(\trace  \exp( \H + \rev{\theta} \A(\y_r)/r))
 $$
 with $\y_r \sim \nu_m^{\otimes m}$. Then by successive conditioning (as in the proof of Proposition \ref{prop:repeated-dpp-proba-conj}), we obtain 
 $$
\Pbb(\lambda_{min}(\frac{1}{r} \sum_{i=1}^r \A(\x_i)) < t ) \le \inf_{\theta <0} e^{-\theta t} \Ebb(\trace  \exp( \sum_{i=1}^r \theta  \A(\y_i)/r)),
$$
where the $\y_i$ are i.i.d. samples from $\nu_m^{\otimes m}$, and   we proceed with the classical proof of matrix Chernoff inequality for sums of i.i.d. matrices (Theorem \ref{th:matrix-chernoff}).  
\end{remark}

\revb{
\begin{remark}[Complexity]\label{complexity-repeated-dpp}
Let us assume that $\mu$ is a discrete measure with $N$ atoms. From Remark \ref{rem:dpp-sampling-complexity}, we know that getting $r$ independent samples from the DPP distribution costs $O(r (m^3 + Nm))$ (up to log factors). With $r \sim \log(m)$, this results in a cost in $O(m^3 + Nm)$  (up to log factors). 
On the other hand, getting $n$ i.i.d. samples from $\nu_m$ costs $O(Nn)$. Then the  subsampling algorithm from \cite{bartel2023} to obtain a subsample of size $O(m)$ costs $O(n m^3)$. With $n \sim m\log(m)$, this yields a total cost in $O(m^4 + Nm)$ (up to log factors).
This shows the advantage of using repeated DPP  to  directly obtain a sample of size $O(m)$, compared to using  i.i.d. sampling and subsampling.
\end{remark}
}

\subsection{Independent repetitions of volume sampling $\gamma^\nu_{\bar n}$}

We here consider weighted least-squares projection using a set of samples gathering independent samples from the volume sampling distribution $\gamma^{  \nu}_{\bar n}$ with $\bar n \ge m$ and $\nu = \revb{w^{-1}} \nu$ with $ \revb{w^{-1}} = \alpha \revb{w_m^{-1}} + (1- \alpha) h $, where the \revb{probability} density $h$ is chosen according to some prior assumption on the target function class. 

We consider $r$ i.i.d. samples $\x_k = (x_{1,k} , \hdots , x_{\bar n,k}) \in \Xc^{\bar n}$ from $\gamma_{\bar n}^{  \nu}$ and the corresponding weighted least-squares minimization with $n= \bar n r$ points. 
The empirical Gram matrix can be written 
$$
\G^w = \frac{1}{r} \sum_{k=1}^r \G^{w}(\x_k)
$$
where the $\G^{w}(\x_k)$ are i.i.d. matrices with expectation $\I$ and spectral norm bounded by $ \alpha^{-1} m$.
We start by providing results in expectation. 
\begin{proposition}\label{prop:vs-variance_projection_repeated}
Let $\x = (x_1 , \hdots, x_n)$ be drawn from the distribution $(\gamma_{\bar n}^\nu)^{\otimes r}$ with $\nu = \revb{w^{-1}}\mu$ and $\revb{w^{-1}}\ge \alpha \revb{w_m^{-1}}$. Assume we use weighted least squares with weight function $w$.  Let $S_\delta = \{\lambda_{min}(\G^w) \ge 1-\delta\}$ with $0< \delta<1.$
Then for any $f\in L^2_\mu$,  it holds
$$
\Ebb(\Vert \hat P_{V_m}  f\Vert^2 \vert S_\delta) \le \Pbb(S_\delta)^{-1} (1-\delta)^{-1}  \beta \Vert f \Vert^2,
$$  
and 
\begin{align*}
  \Ebb(\Vert \hat P_{V_m} f \Vert^2 | S_\delta) &\le \Pbb(S_\delta)^{-1} (1-\delta)^{-2} (   \frac{m}{  n} \alpha^{-1} ( \beta + \xi m \alpha^{-1} + \beta \xi n)   \Vert f \Vert^2  +   (1-\xi + \xi/r) \Vert P_{V_m} f \Vert^2),
\end{align*}
with $\beta =1 + (\alpha^{-1}-1) \frac{m}{\bar n}$,  
 $\xi    = 0$ if $\nu = \nu_m$, or $\xi=1$ in the case $\nu\neq \nu_m$.
\end{proposition}
\begin{proof}
The marginal distributions  of $\x$ are all equal to $\tilde \nu = \revb{\tilde w^{-1}} \mu$ with $\revb{\tilde w^{-1}}  \le \beta \revb{w^{-1}}.$ Then the first inequality directly follows from Lemma \ref{lem:empirical_projection_iid}. 
Then following the proof of Proposition \ref{prop:vs-variance_projection}, we obtain 
\begin{align*}
 \Ebb(\Vert   \hat P_{V_m} f   \Vert_2^2 | S_\delta) &\le  \Pbb(S_\delta)^{-1} (1-\delta)^{-2} \Ebb(\Vert \b \Vert_2^2),
\end{align*}
with 
$\b = \frac{1}{n} \Phi^w(\x)^T f^w(\x) = \frac{1}{r} \sum_{k=1}^r \b(\x_k),$ where $\b(\x_k) = \frac{1}{\bar n}  \Phi^w(\x_k)^T f^w(\x_k) $ and $\x_k \sim \gamma_{\bar n}^\nu$. The $\b(\x_k)$ being i.i.d., it holds 
$$
\Ebb(\Vert \b\Vert_2^2) = \frac{1}{r} \Ebb(\Vert \b(\z) \Vert_2^2) +  \frac{r-1}{r}\Vert \Ebb(\b(\z))\Vert^2_2
$$
with $\z \sim \gamma_{\bar n}^\nu.$ Using Lemma \ref{lem:vs-variance_quasi_projection}, we have
\begin{align*}\Ebb(\Vert \b(\z) \Vert_2^2) \le 
 \frac{m}{\bar n} \alpha^{-1} ( \beta + \xi m \alpha^{-1})   \Vert f \Vert^2  +    \Vert P_{V_m} f \Vert^2,
\end{align*}
with $\beta =1 + (\alpha^{-1}-1) \frac{m}{\bar n}$,  
 $\xi  = 0$ if $\nu = \nu_m$ and $\xi=1$ in the case $\nu\neq \nu_m$. When $\nu = \nu_m$, it holds 
 $\Vert \Ebb(\b(\z))\Vert^2_2 = \Vert \Ebb_{x \sim \nu_m}(\varphib(x) f(x) \revb{w_m(x)}) \Vert^2_2 = \Vert P_{V_m} f  \Vert^2$. 
 When $\nu\neq \nu_m$, we have 
  \begin{align*}\Vert \Ebb(\b(\z))\Vert^2_2  &= 
 \Vert \Ebb_{x \sim \tilde \nu}(\varphib(x) f(x) \revb{ w(x)}) \Vert^2_2 \\
 &\le \Ebb_{x \sim \tilde \nu}(\Vert \varphib(x) \Vert_2^2 f(x)^2 \revb{ w(x)^2} ) \\
 &\le m \alpha^{-1} \beta \revb{ \int f(x)^2 d\mu(x) } =m \alpha^{-1} \beta \Vert f\Vert^2  .
 \end{align*}  
 Gathering the above results, we obtain 
  \begin{align*}
  \Vert \Ebb(\b(\z))\Vert^2_2 \le  \frac{m}{  n} \alpha^{-1} ( \beta + \xi m \alpha^{-1})   \Vert f \Vert^2  +   (1-\xi + \xi/r) \Vert P_{V_m} f \Vert^2 + m\alpha^{-1} \beta \xi \Vert f\Vert^2,
  \end{align*}
  which ends the proof.
\end{proof}
 
\begin{theorem}\label{th:repeated-vs-quasi-optimality-expectation}
Let $r\in \Nbb$, $\bar n\ge m$ and $n= \bar n r$.  Assume $\x$ is drawn from the distribution $(\gamma_{\bar n}^\nu)^{ \otimes r}$ with $\nu = \revb{w^{-1}} \mu$ such that  $ \revb{w^{-1}}\ge \alpha  \revb{w_m^{-1}}$, and assume we use weighted least-squares with weight function $w$. Letting $S_\delta = \{\lambda_{min}(\G^{w_m}(\x) ) \ge 1-\delta\}$, it holds 
$$
\Ebb(\Vert f - \hat f_m \Vert^2 \vert S_\delta) \le (1 +  \Pbb(S_\delta)^{-1} (1-\delta)^{-1} \beta ) \inf_{g\in V_m} \Vert f - g \Vert^2,
$$
and  
$$
\Ebb(\Vert f - \hat f_m \Vert^2 \vert S_\delta) \le (1 +  
\Pbb(S_\delta)^{-1} (1-\delta)^{-2} (   \frac{m}{  n} \alpha^{-1} ( \beta + \xi m \alpha^{-1} + \beta \xi n) 
)) \inf_{g\in V_m}  \Vert f - g \Vert^2.
$$
with $\beta =1 + (\alpha^{-1}-1) \frac{m}{\bar n}$,  
 $\xi    = 0$ if $\nu = \nu_m$, or $\xi=1$ in the case $\nu\neq \nu_m$.
 \end{theorem}
 \begin{proof}
 This simply results from Lemma \ref{lem:quasi-optimality-conditional-expectation} and Proposition \ref{prop:vs-variance_projection_repeated}.
 \end{proof}

We next provide a result in probability and another result in expectation (without conditioning) under the assumption that the target function $f$ in some subspace $H$. 
\begin{theorem}\label{th:repeated-vs-quasi-optimality-H}
Assume that $f\in H$, with $H$ satisfying \eqref{H-embedding} and $\eqref{H-hcondition}$, with $h$ a \revb{probability} density with respect to $\mu$. 
Assume that $\x = (x_1,\hdots,x_n)$ is drawn from $(\gamma_{\bar n}^\nu)^{ \otimes r}$ with $\nu =  \revb{w^{-1}} \mu$ and $ \revb{w^{-1}} =   \alpha  \revb{w_m^{-1}} +  ({1-\alpha}) h$, and  we use weighted least-squares with weight function $w$.
  Then it holds 
\begin{align*}
\Vert f - \hat f_m \Vert &\le \left( C_H+ (1-\delta)^{-1/2} (1-\alpha)^{-1/2}   \right) \inf_{g\in V_m}\Vert f - g \Vert_H  
\end{align*}
with probability at least $\Pbb(S_\delta),$ where $S_\delta = \{\lambda_{min}(\G^{w}(\x) ) \ge 1-\delta\}.$
If 
$$
\bar n \ge m + c_\delta^{-1} \alpha^{-1} m \log(nm^2)
$$
 it holds 
\begin{align*}
\Ebb(\Vert f - \hat f_m \Vert^2) &\le \left( 2C_H^2+2  (1-\alpha)^{-1} r(1+ (1-\frac{m}{\bar n})^{-1} (1-\delta)^{-1}) \right) \inf_{g\in V_m}\Vert f - g \Vert_H ^2
\end{align*}
\end{theorem}
\begin{proof}
The first inequality is deduced from Lemma \ref{lem:near-optimality-H} with $\zeta = (1-\alpha)^{-1}$. 
For the second inequality, we use Lemma  \ref{lem:near-optimality-H} with $\zeta = (1-\alpha)^{-1}$ and the fact that for any $1\le k \le r$, it holds 
$\Ebb(\lambda_{min}(\G^w(\x))^{-1}) \le r 
\Ebb(\lambda_{min}(\G^w(\x_k))^{-1}) 
\le r(1+ (1-\frac{m}{\bar n})^{-1} (1-\delta)^{-1})$, where the last inequality comes from  Lemma \ref{lem:vs-lambdamin}.
\end{proof} 
The first statement of Theorem \ref{th:repeated-vs-quasi-optimality-H} is a $H\to L^2_\mu$ quasi-optimality result in probability. The second statement is a $H\to L^2_\mu$ quasi-optimality property in expectation, without conditioning  the sample to satisfy the event $S_\delta$. 

Again it remains to control the probability of the event $S_\delta$. 
In fact, the distribution of $\G^w$ is the one of the average of $r$ independent Gram matrices associated with $\gamma_m$, and $r (\bar n - m)$ rank-one matrices associated with i.i.d. samples from $\nu$. All these $r (\bar n - m+1)$ matrices have spectral norm bounded by $\alpha^{-1} m$. 
Therefore, from matrix Chernoff inequality (Theorem \ref{th:matrix-chernoff}), we deduce that 
$$
\Pbb( \lambda_{min}(\G^w) > 1-\delta ) \le m \exp(- \frac{r (\bar n - m+1) c_\delta^{-1} \alpha}{m})  ,
$$
and we can conclude that 
$$
\Pbb( \lambda_{min}(\G^w) > 1-\delta ) \le \eta
$$
whenever $r (\bar n - m+1) \ge  c_\delta^{-1} m \alpha^{-1} \log( m \eta^{-1})$, or the condition 
$n \ge  \frac{\bar n}{\bar n - m+1} c_\delta^{-1} m \alpha^{-1} \log( m \eta^{-1})$ on the total number of samples. 
For, e.g., $\bar n = 2m$, we obtain a condition $n \ge 2c_\delta^{-1} m \alpha^{-1} \log(m\eta^{-1})$, that is suboptimal (by a factor $2$) compared to i.i.d. sampling. Here, we obtain a complexity in $O(m\log(m))$ similar to i.i.d. but this is essentially due to the presence in $\x$ of $mr$ i.i.d. samples from $\nu$. 
Again, this analysis  does not really exploit the properties of volume sampling.

A better understanding of the distribution of matrices $\G^{w}(\x_k)$ 
allows to improve the above results. 
As for the case of determinantal point processes,  we conjecture that 
\begin{align}
\Pbb(F(\G^{w}(\x_k))>t) \le \Pbb_{ \y \sim \nu^{\otimes \bar n}}(F(\G^{w}(\y)) > t) \label{conj_pvs}
\end{align}
for $t>0$ and $F$ a  real-valued positive, convex and monotonically decreasing function in the Loewner order. 
Under this conjecture, following the proof of Proposition \ref{prop:repeated-dpp-proba-conj},  we would obtain the following concentration result, similar to the case of $n = m\bar n $ i.i.d. sampling from $\nu = w \mu$.
\begin{proposition}\label{prop:pvs-first}
Let $r\in \Nbb$, $\bar n \ge m$ and $n= \bar n r$. Assume $\x = (\x_1, \hdots,\x_r) \sim (\gamma^\nu_{\bar n})^{\otimes r}$ with $\gamma^\nu_{\bar n}$ a distribution over $\Xc^{\bar n}$ satisfying \eqref{conj_pvs}.
Then it holds 
$$\Pbb(\lambda_{min}(\G^{w}(\x) < 1-\delta)  \le m \exp(-  \frac{c_\delta n}{m} ).$$
\end{proposition}
 \begin{remark}
The assumption \eqref{conj_pvs} in Proposition \ref{prop:pvs-first} could be replaced by a weaker condition of the form  \eqref{ass_dpp_bis}, or an alternative condition of the form \eqref{ass_dpp_ter}. \end{remark}

We can avoid any  conjecture on volume sampling and still obtain a result similar to Proposition \ref{prop:pvs-first} by assuming that the DPP distribution $\gamma_m$ satisfies  the conjecture   \eqref{conj_2} (or one of the two conjectures  \eqref{ass_dpp_bis} or \eqref{ass_dpp_ter}), and further assuming that $\revb{w_m} \le \xi_m \revb{w}$ for some constant $\xi_m$ (possibly depending on $m$).
\begin{proposition}\label{prop:vs-proba-under-conj}
Let $\bar n \ge m$, $r \in \Nbb$ and $n = \bar n m$.  Let  $\y \sim \nu^{\otimes (n-mr)}$ with $\nu=\revb{w^{-1}} \mu$ such that $\revb{w_m} \le \xi_m \revb{w}$. Let $\z = (\z_1,\hdots,\z_r) \sim \gamma_m^{\otimes r}$ where   $\gamma_m$ is a distribution over $\Xc^m$ 
satisfying either   \eqref{conj_2}, \eqref{ass_dpp_bis} or \eqref{ass_dpp_ter}.
Letting $\x = (\y , \z)$, it holds 
$$\Pbb(\lambda_{min}(\G^{w}(\x)) < (1-\delta) \xi_m^{-1} \frac{m}{\bar n}) \le m \exp(-  c_\delta r ).$$ 
\end{proposition}
\begin{proof}
We have 
$$
\G^w(\x) = \frac{mr}{n} \G^w(\z) +   \frac{n-mr}{n} \G^w(\y) \succeq \frac{mr}{n} \G^w(\z) \succeq \frac{mr}{n} \xi_m^{-1} \G^{w_m}(\z). 
$$
Therefore, using Proposition \ref{prop:repeated-dpp-proba-conj} with assumption \eqref{conj_2} or the alternative assumptions \eqref{ass_dpp_bis} or \eqref{ass_dpp_ter} (see remarks \ref{rem:ass-lmin} and \ref{rem:ass-trace-exp}), it holds  
$$\Pbb(\lambda_{min}(\G^{w}(\x)) < (1-\delta) \xi_m^{-1} \frac{mr}{n}) \le \Pbb(\lambda_{min}(\G^{w_m}(\z)) < 1-\delta) \le m \exp(- c_\delta r ).$$ 
\end{proof}
Provided $n= \bar n r \ge \bar n c_{\delta}^{-1} \log(m \eta^{-1})$, it then holds $\Pbb(\lambda_{min}(\G^{w}(\x))^{-1} < (1-\delta')^{-1}) \ge 1-\eta$ with $(1-\delta')^{-1} = (1-\delta)^{-1} \xi_m \frac{\bar n}{m}$. Theorem \ref{th:repeated-vs-quasi-optimality-expectation} then gives 
$$
\Ebb(\Vert f - \hat f_m \Vert^2 \vert S_{\delta'}) \le (1 + (1-\eta)^{-1} \beta \xi_m \frac{\bar n}{m}) \inf_{v\in V_m} \Vert f - v \Vert^2,
$$
which is a quasi-optimality result only if $\xi_m$ is uniformly bounded with $m$.
\begin{remark}
Note that   if $V_m$ contains the constant function and $\revb{w^{-1}}  = \alpha \revb{w_m^{-1}}  + (1-\alpha) h$ with $h=1$, then $\revb{w_m^{-1}}  \ge \frac{1}{m}$ and therefore, $\revb{w_m}  \le \xi_m \revb{w} $ with $\xi_m = \alpha + (1-\alpha) m \le m$.
\end{remark}
\reva{Note that the above theoretical results  only show that repeated volume sampling is not worse than i.i.d. sampling, but numerical experiments reveal that repeated volume sampling clearly outperforms i.i.d. sampling. To be understood, this would again require other tools for analyzing the concentration of $\G^{w}$.}

\section{Numerical experiments}\label{sec:experiments}

We consider two simple cases of polynomial approximation where $V_m$ is the space of polynomials of degree $m-1$, with either $\Xc = [-1,1]$ equipped with the uniform measure $\mu \sim U(-1,1)$, or  $\Xc = \Rbb$ equipped with the standard gaussian measure $\mu \sim \Nc(0,1)$. We compare (weighted) least-squares methods using (i) $n$ i.i.d. samples from $\mu$, (ii) $n$ i.i.d. samples from $\nu_m = \revb{w_m^{-1}} \mu$, (iii) $n$ samples drawn from  volume-rescaled sampling distribution $\gamma_n^{\nu_m}$ (that is equivalent to $m$ samples from $\gamma_m$ and $n-m$ i.i.d. samples from $\nu_m$), and (iv) $n$ samples from independent repetitions of  projection DPP distribution $\gamma_m$. In the latter case, we perform $r = \lceil n/m \rceil$ i.i.d. samples from $\gamma_m$ and keep the first $n$ points.

\revb{Concerning  sampling, we  systematically approximated measures $\mu$ by discrete measures with $N=2000$ atoms, and then relied on exact samplers for discrete distributions. This approximation has no impact on the qualitative results below.}

For the case of the uniform distribution over $[-1,1]$, \Cref{fig:gram-lmin-comparison-uniform} shows estimations of the probability to satisfy   $S_\delta = \{\lambda_{min}(\G^w) \ge 1 - \delta \}$  
for the different sampling methods, as a function of $m$ and $n$. We observe a clear superiority of optimal i.i.d. sampling compared to classical i.i.d. sampling. We also observe that volume-rescaled sampling brings only a small improvement over optimal i.i.d. sampling. Finally, we observe that using independent repetitions of $\gamma_m$ clearly outperforms volume-rescaled sampling. For i.i.d. optimal sampling and volume-rescaled sampling, we can observe from the figure the condition $n \sim m\log(m)$ to ensure that  $S_\delta$ 
is satisfied with probability $1/2$. For repeated DPP, we observe
a condition better than $n \sim m\log(m)$, or at least with a much better constant than with other approaches, that is unfortunately not explained by our theoretical findings.   

\begin{figure}[h]\centering 
\subfloat[i.i.d. $\mu$]{\includegraphics[width=0.3\linewidth]{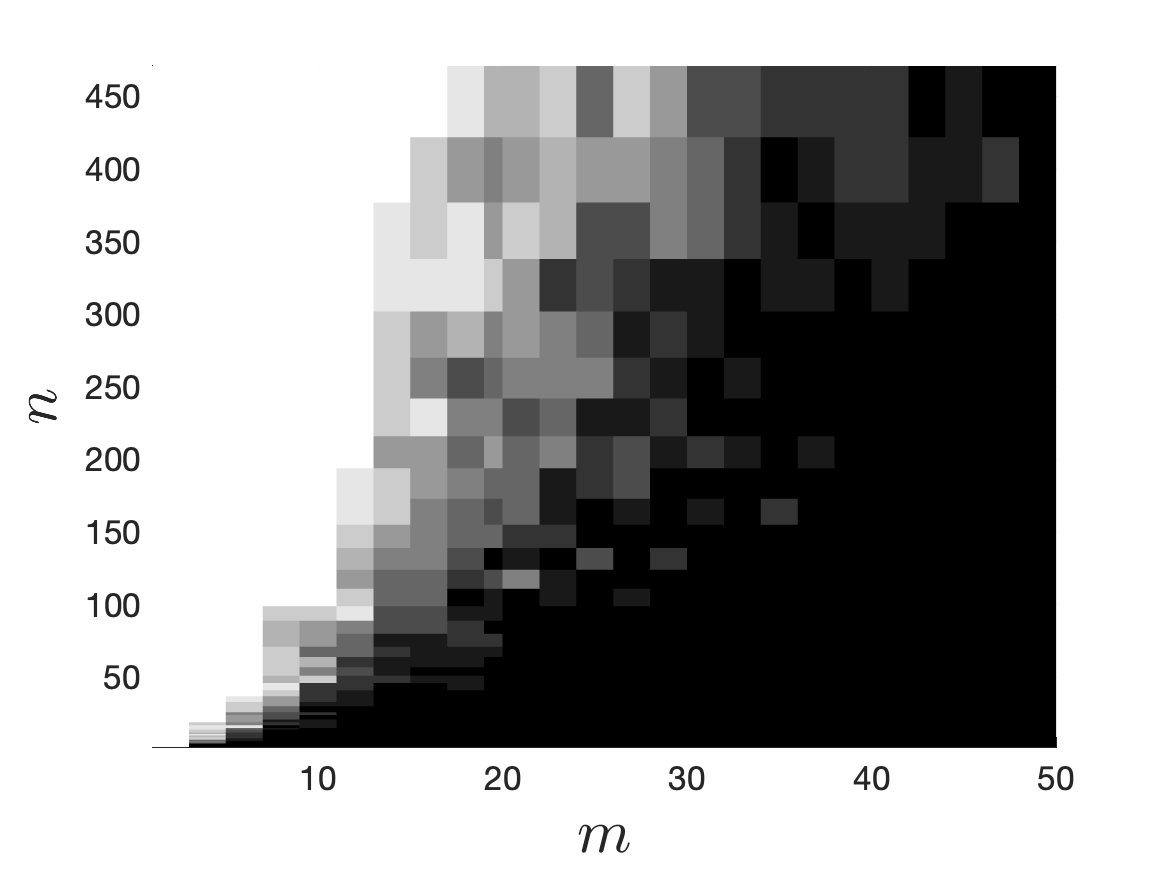}}
\subfloat[i.i.d. $\nu_m$]{\includegraphics[width=0.3\linewidth]{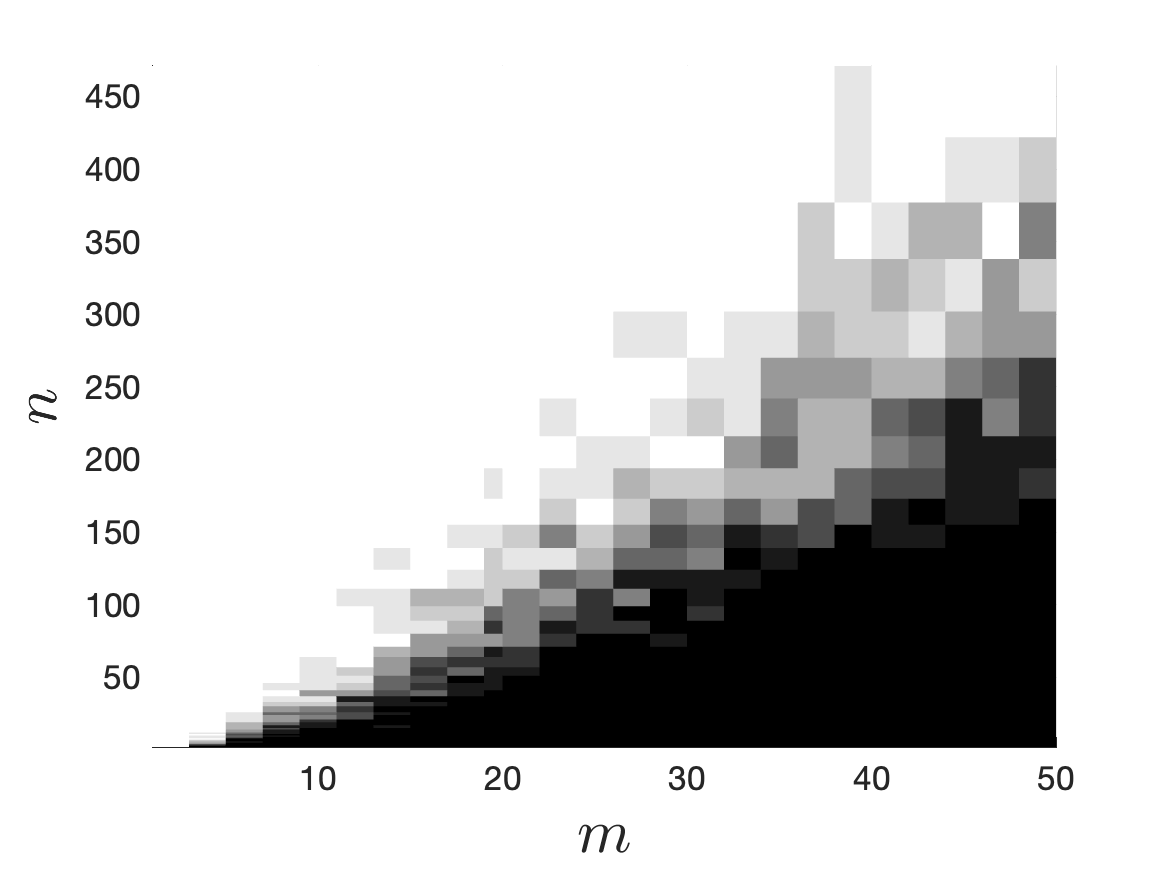}} 
\\
\subfloat[$\gamma^{\nu_m}_n$]{\includegraphics[width=0.3\linewidth]{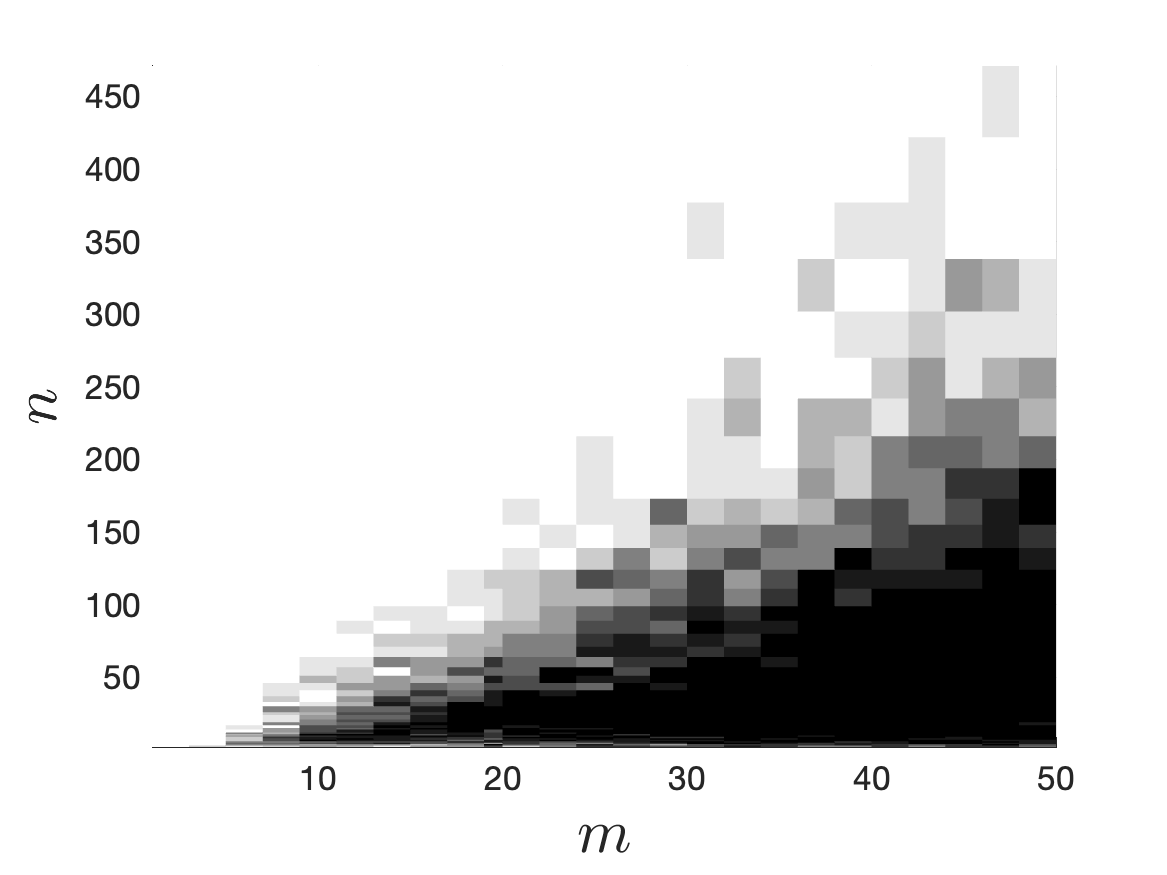}}
\subfloat[$ \gamma_m^{\otimes \lceil n/m \rceil}$]{\includegraphics[width=0.3\linewidth]{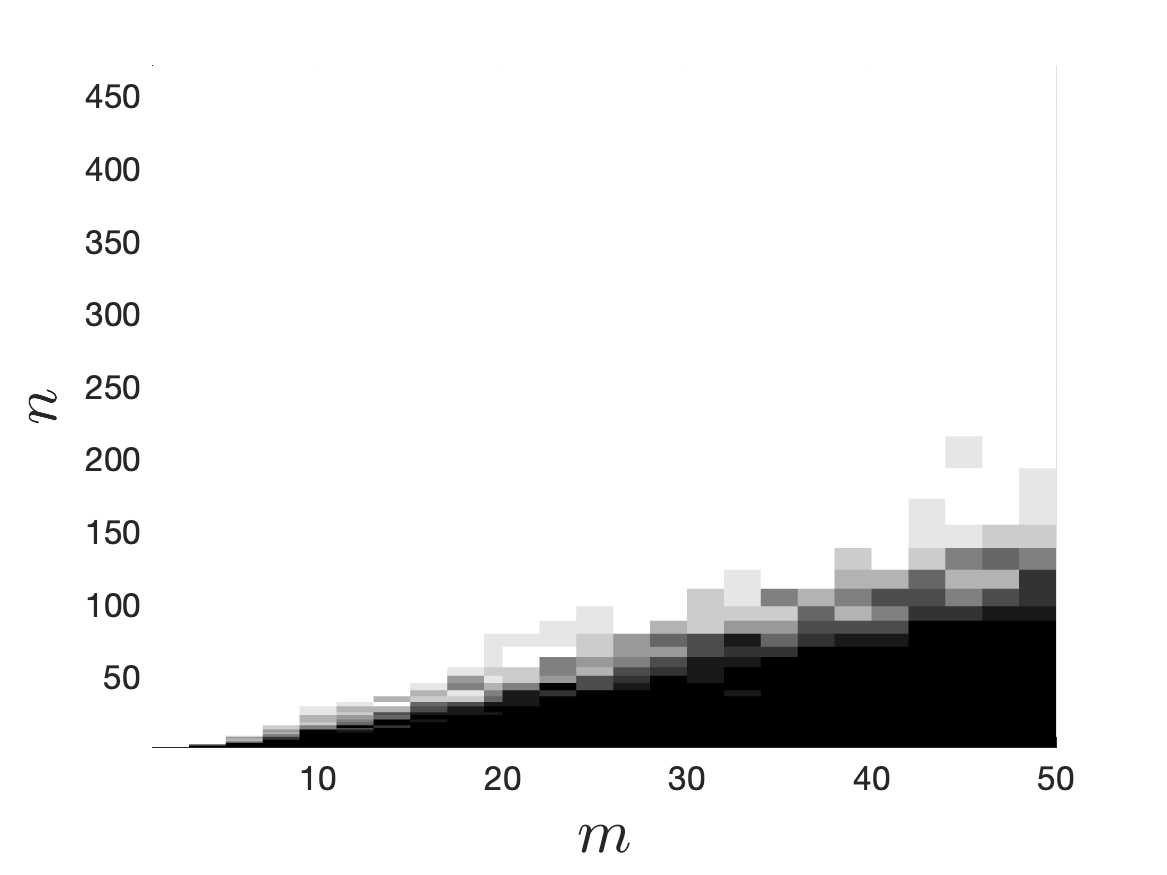}}
\caption{For $V_m$ the space of polynomials of degree $m-1$ and $\mu$ the  uniform measure on $[-1,1]$, we plot $\Pbb(\lambda_{min}(\G^w) \ge 1/4)$  as a function of dimension $m$ and number of samples $n.$ Probability estimates go from 0 (black) to 1 (white).}\label{fig:gram-lmin-comparison-uniform}
\end{figure}
%
\Cref{fig:gram-lmin-comparison-gaussian} illustrates the same quantities for the case of the standard gaussian measure $\mu$ over $\Xc = \Rbb$. We draw essentially the same conclusions. We notice in this case the very poor performance of classical sampling from $\mu$. 

\begin{figure}[h]\centering 
\subfloat[i.i.d. $\mu$]{\includegraphics[width=0.3\linewidth]{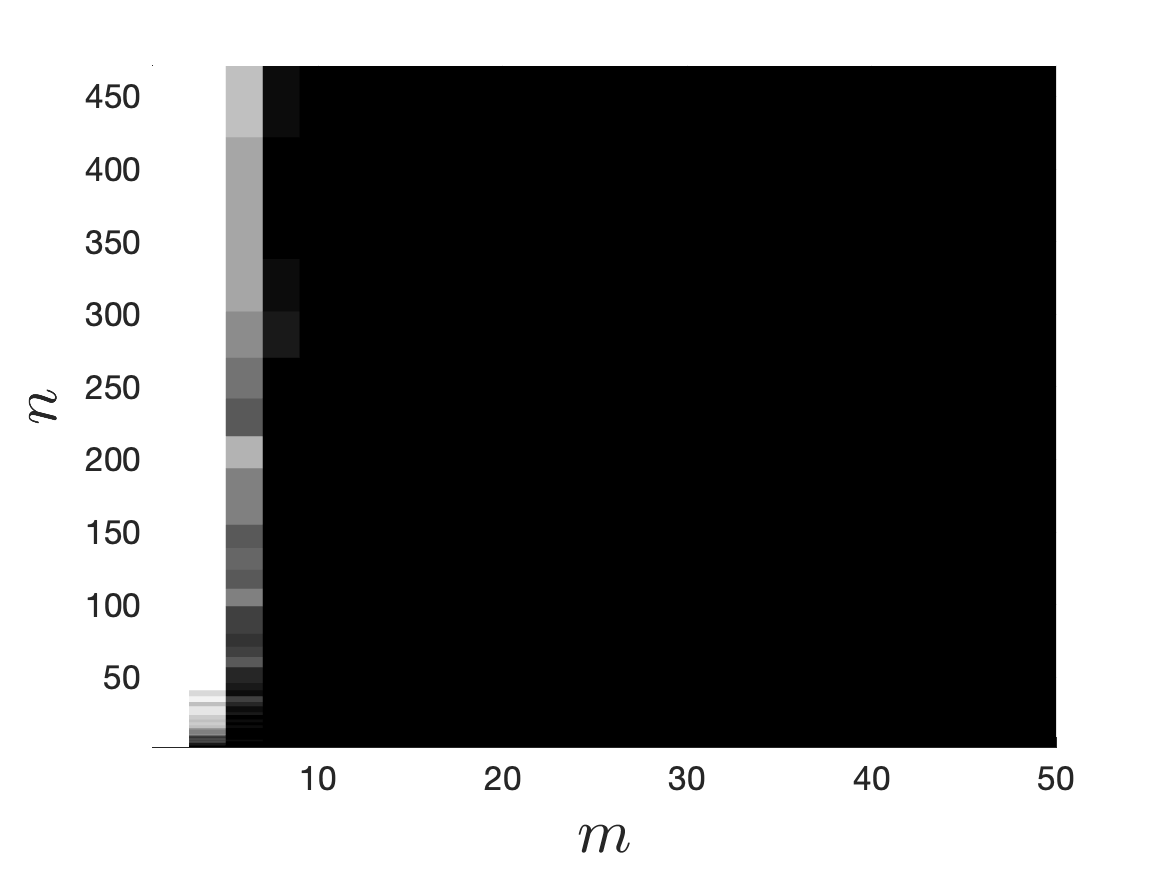}}
\subfloat[i.i.d. $\nu_m$]{\includegraphics[width=0.3\linewidth]{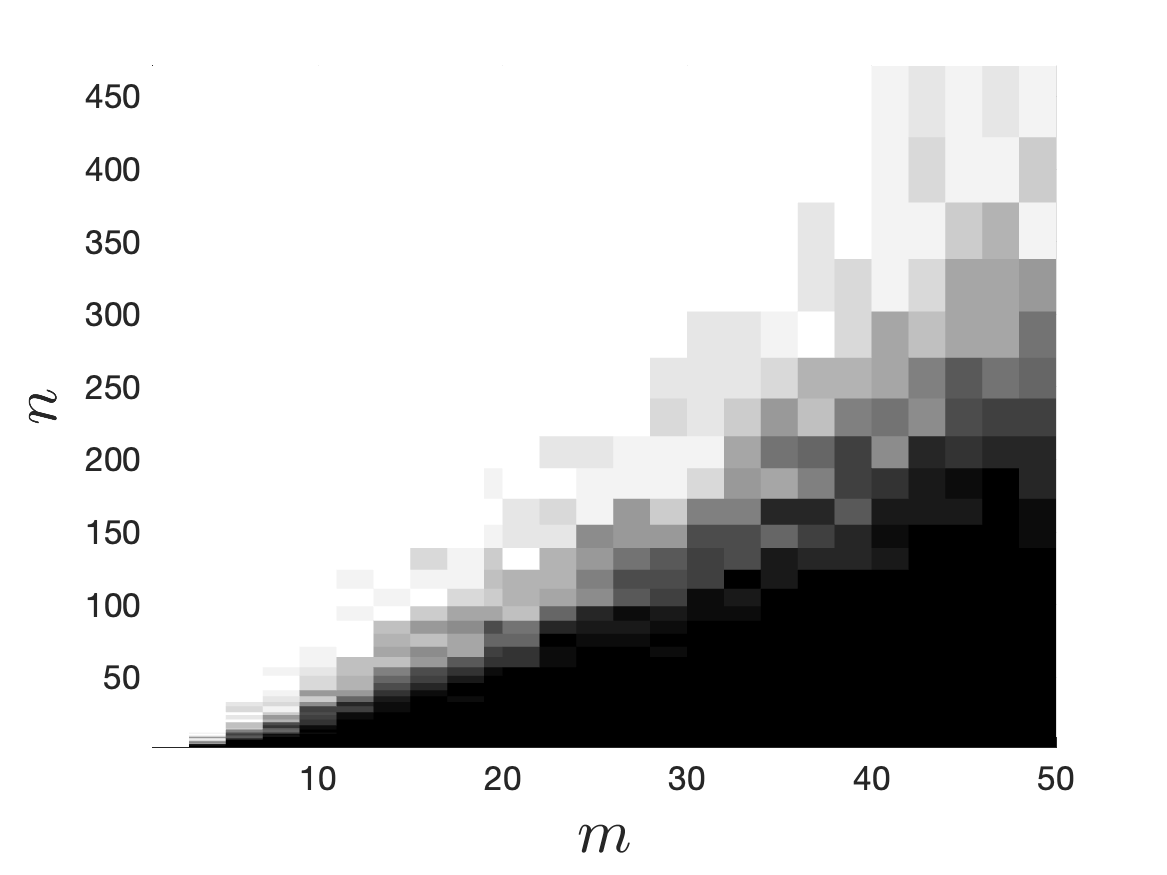}} 
\vspace{-1em}\\
\subfloat[$\gamma_m$ + $n-m$ i.i.d. $\nu_m$]{\includegraphics[width=0.3\linewidth]{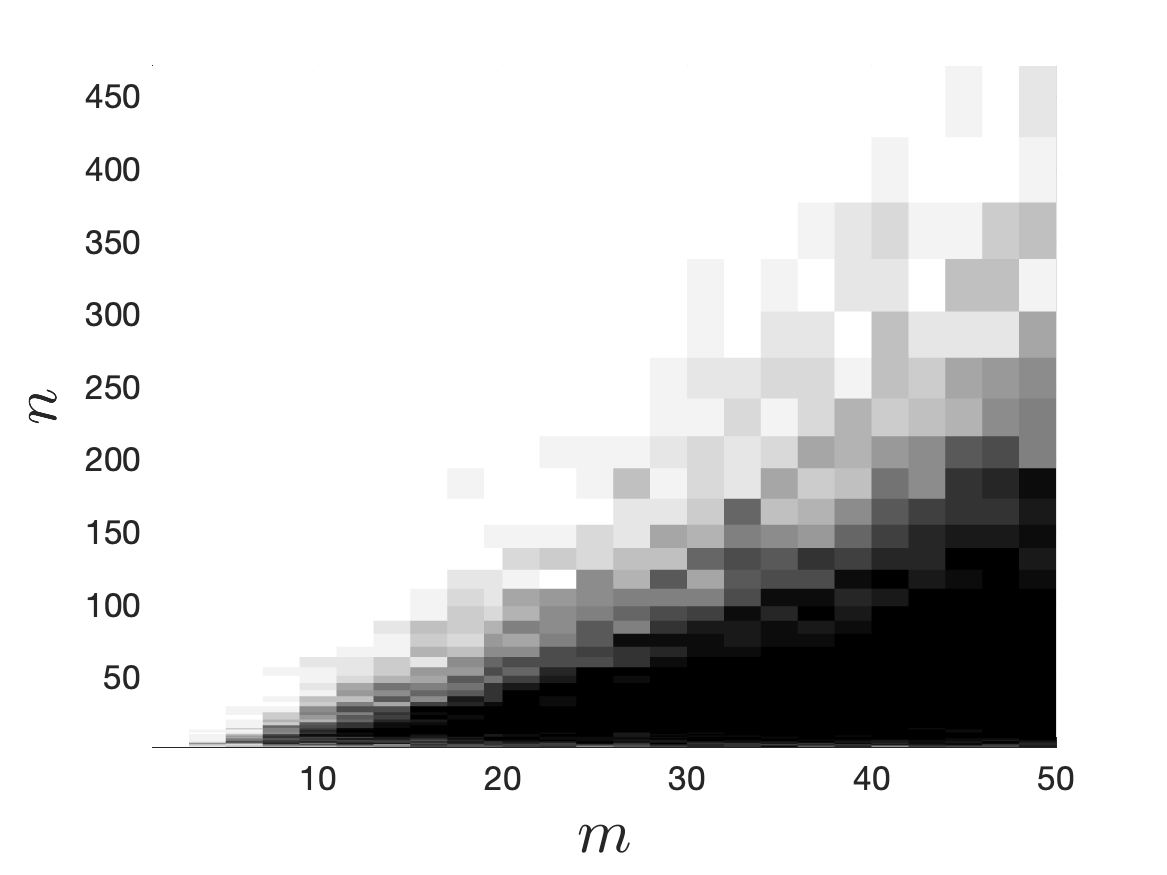}}
\subfloat[repeated $ \gamma_m$]{\includegraphics[width=0.3\linewidth]{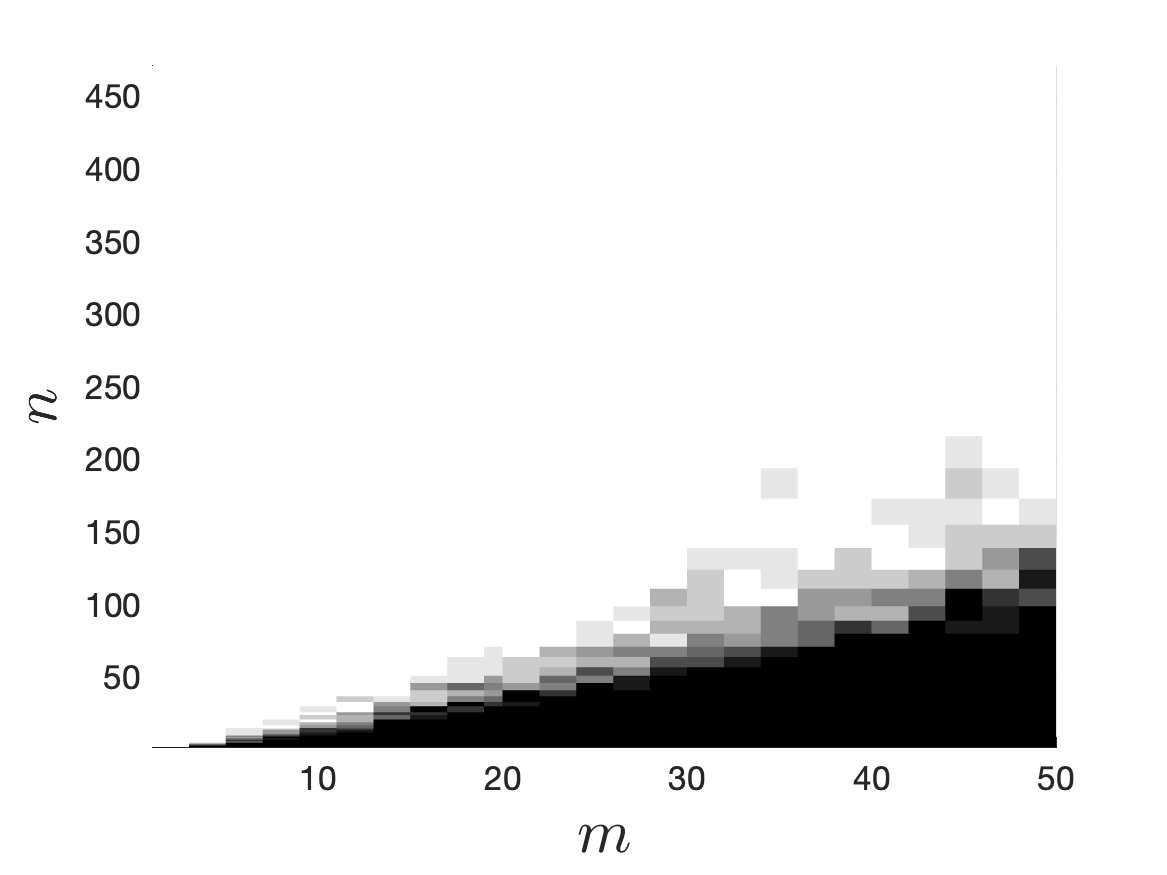}}
\caption{For $V_m$ the space of polynomials  of degree $m-1$ and $\mu$ the standard gaussian measure on $\Rbb$, we plot $\Pbb(\lambda_{min}(\G^w) \ge 1/4)$  as a function of dimension $m$ and number of samples $n.$ Probability estimates go from 0 (black) to 1 (white).}\label{fig:gram-lmin-comparison-gaussian}
\end{figure}
%

From now on, we consider the approximation of the function $f(x) = (1+2x^2)^{-1}$ on $\Rbb$ equipped with the standard Gaussian measure $\mu$, and $V_m$ the space of polynomials of degree $m-1$. 
On \Cref{fig:histogram-L2-error}, we plot the histograms of the logarithm of the $L^2_\mu$ relative error, $ \log( \Vert f - \hat f_m \Vert / \Vert f \Vert)$, for the different sampling strategies and using $n = r m $ for different $r$. For small oversampling ($r = 2$), we observe the benefit of volume-rescaled sampling over i.i.d. sampling, which shifts the distribution towards small values. We also observe  a clear superiority of repeated DPP $\gamma_m^{\otimes r}$, which is   further improved by conditioning to satisfy the event $S_{\delta}$ with $\delta=3/4$.
For large oversampling ($r=10$), the histograms  are roughly similar. We only observe a slight benefit of conditioned repeated DPP over the other methods, even over i.i.d. optimal sampling. 

\begin{figure}
\centering 
\subfloat[$n=2m$]{\includegraphics[width=.45\linewidth]{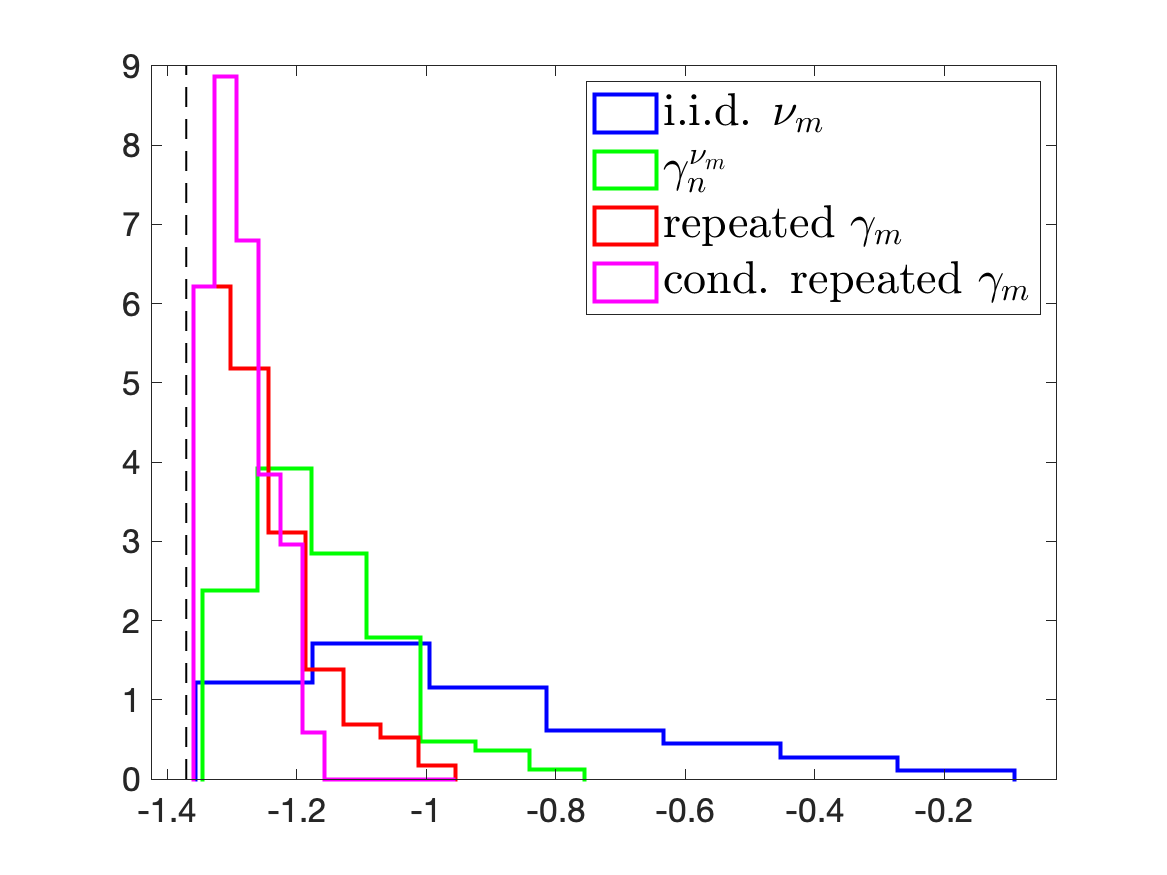}}
\\\subfloat[$n=5m$]{\includegraphics[width=.45\linewidth]{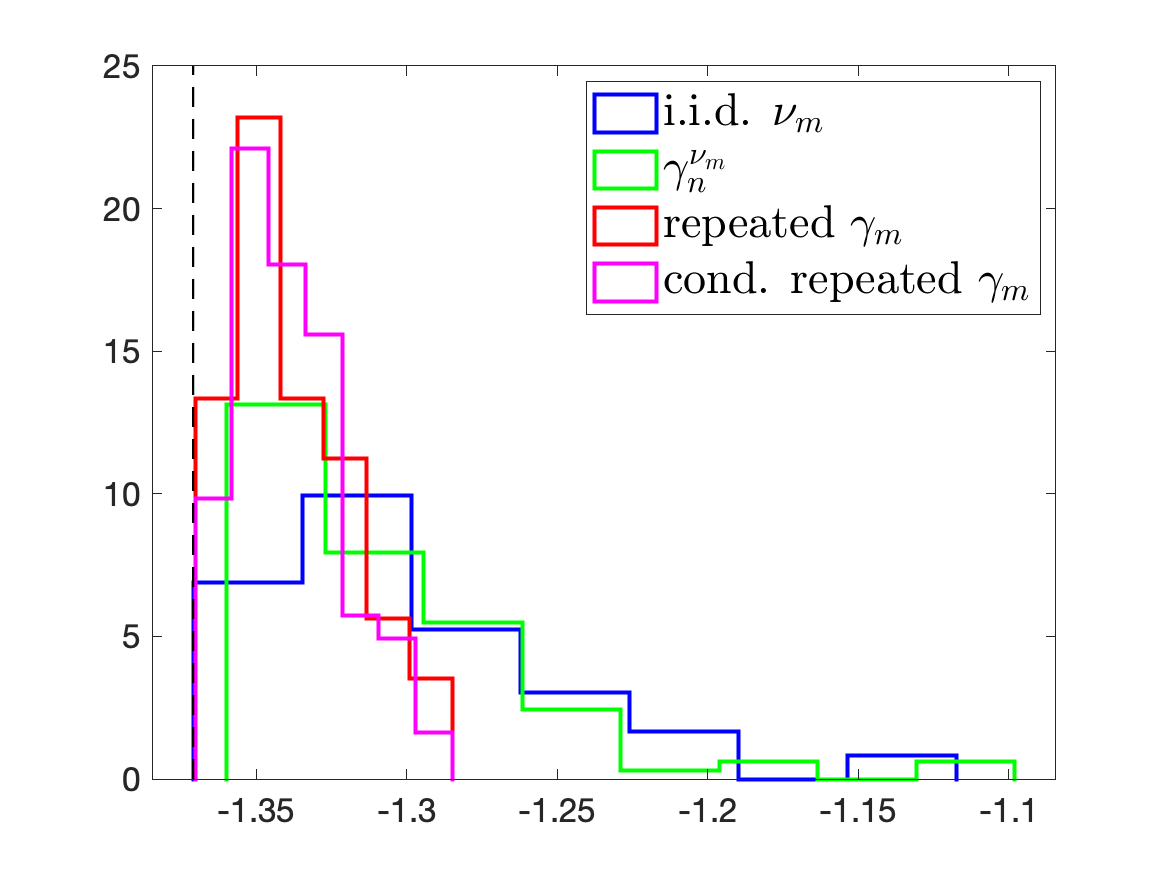}}
\subfloat[$n=10m$]{\includegraphics[width=.45\linewidth]{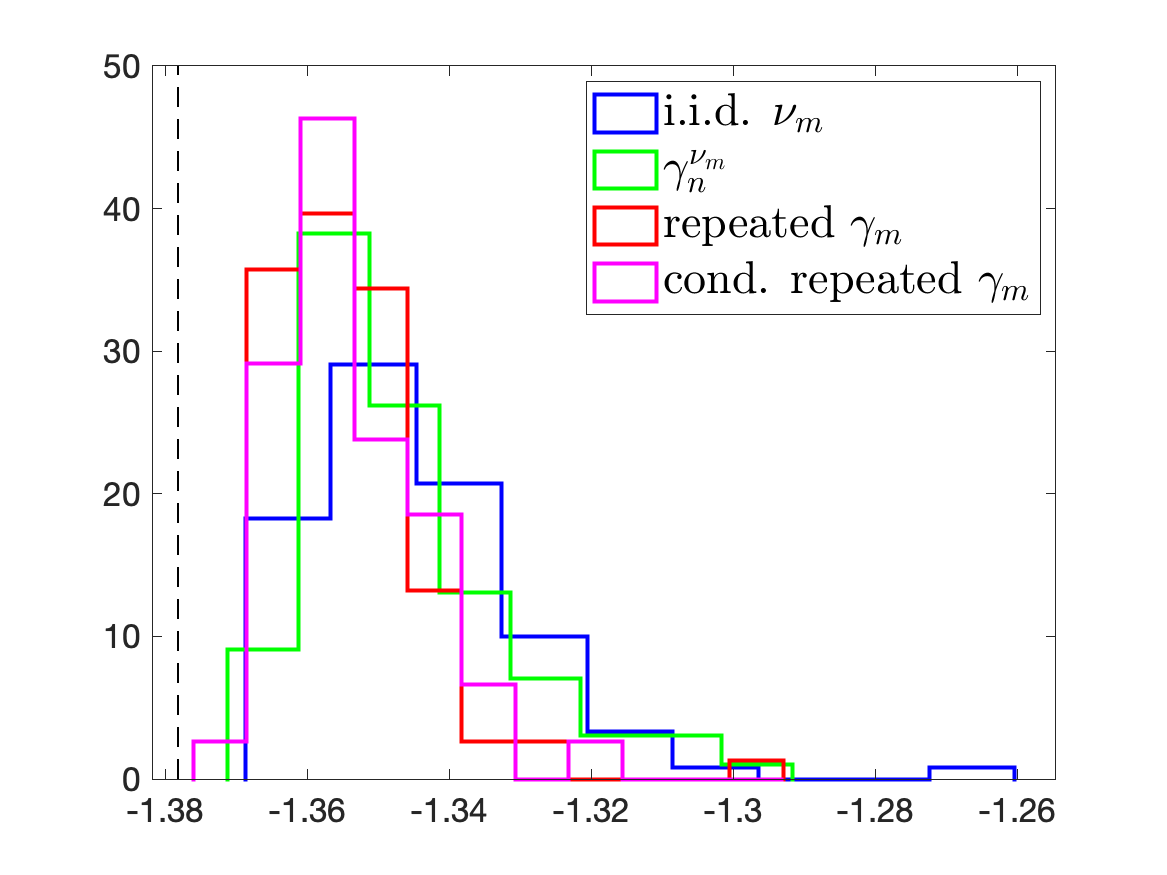}}
\caption{For $V_m$ the space of polynomials  of degree $m-1$ and $\mu$ the standard gaussian measure on $\Rbb$, we plot the histogram of $ \log( \Vert f - \hat f_m \Vert / \Vert f \Vert)$ using $n = r m $ samples from $\nu_m^{\otimes n}$ (blue),  the volume-rescaled sampling distribution $\gamma_n^{\nu_m}$ (green),  repeated DPP $\gamma_m^{\otimes r}$ (red), or  repeated DPP conditioned to satisfy $S_{\delta}$ with $\delta = 3/4$ \rev{(magenta)}.}
\label{fig:histogram-L2-error}
\end{figure}

\Cref{tab:L2error} shows the expected relative error $\Ebb(\Vert f - \hat f_m \Vert^2)^{1/2} / \Vert f \Vert$ and the quantile of $\Vert f - \hat f_m \Vert / \Vert f \Vert$ of level $95\%$. We first observe the catastrophic results for classical i.i.d. sampling from $\mu.$ 
For small oversampling ($n=2m$), we observe on both criteria a clear benefit of volume-rescaled and repeated DPP over i.i.d. \rev{optimal} sampling. In terms of expected error, we observe a slight improvement of repeated DPP compared to volume-rescaled sampling. However, concerning the quantile, we observe a clear superiority of repeated DPP over volume-rescaled sampling. For the same number of samples, this quantile is divided by up to a factor 2. 
For larger oversampling ($n=5m$), i.i.d. performs much better and gets closer to the performance of volume-rescaled sampling and repeated DPP. 
\begin{table}
\begin{subtable}{\textwidth}
$$
  \begin{array}{|c|c|c|c|c|c|c|c|}
     \hline m & \text{best} & \text{i.i.d. $\mu$} & \text{i.i.d. $\nu_m$} & \revb{\text{i.i.d. $\nu_m$ (cond.)}} & \text{$\gamma_n^{\nu_m}$} & \text{$\gamma_m^{\otimes n/m}$} & \text{$\gamma_m^{\otimes n/m}$ (cond.)} \\ \hline 10 & 1.3e-01 & 8.4e+02 & 4.7e-01 & 1.7e-01 & 2.0e-01 & 1.7e-01 & 1.6e-01\\ \hline20 & 5.1e-02 & 3.7e+07 & 1.5e-01 & 6.4e-02 & 8.2e-02 & 6.7e-02 & 6.3e-02\\ \hline30 & 2.6e-02 & 3.3e+10 & 2.5e-01 & \times & 4.1e-02 & 3.5e-02 & 3.2e-02\\ \hline40 & 1.4e-02 & 1.2e+10 & 7.2e-02 & \times & 2.3e-02 & 1.8e-02 & 1.7e-02\\ \hline50 & 7.9e-03 & 4.9e+09 & 3.4e-02 &\times & 1.3e-02 & 1.1e-02 & 9.6e-03\\ \hline
 \end{array}
  $$
  \caption{$n=2 m$. Expected relative error $\Ebb(\Vert f - \hat f_m \Vert^2)^{1/2} / \Vert f \Vert$.}
  \end{subtable} \medskip
  
  \begin{subtable}{\textwidth}
$$
%
    \begin{array}{|c|c|c|c|c|c|c|c|}
    \hline  m & \text{best} & \text{i.i.d. $\mu$} & \text{i.i.d. $\nu_m$} & \revb{\text{i.i.d. $\nu_m$ (cond.)} } & \text{$\gamma_n^{\nu_m}$} & \text{$\gamma_m^{\otimes n/m}$} & \text{$\gamma_m^{\otimes n/m}$ (cond.)} \\ \hline 10 & 1.3e-01 & 2.8e+03 & 1.2e+00 & 2.3e-01 & 3.4e-01 & 2.5e-01 & 2.2e-01\\ \hline 20 & 5.1e-02 & 1.7e+08 & 4.2e-01 & 8.0e-02 & 1.4e-01 & 9.6e-02 & 8.2e-02\\ \hline30 & 2.6e-02 & 9.3e+10 & 3.2e-01 & 0.0e+00 & 6.6e-02 & 5.5e-02 & 4.0e-02\\ \hline40 & 1.4e-02 & 4.2e+10 & 1.6e-01 & 0.0e+00 & 4.3e-02 & 2.4e-02 & 2.2e-02\\ \hline50 & 7.9e-03 & 1.5e+10 & 1.1e-01 & 0.0e+00 & 2.3e-02 & 1.5e-02 & 1.2e-02\\ \hline \end{array}
  $$
  \caption{$n=2 m$. Quantile of $\Vert f - \hat f_m \Vert / \Vert f \Vert$ of level $95\%$.}
  \end{subtable}\medskip
  
\begin{subtable}{\textwidth}
$$
\begin{array}{|c|c|c|c|c|c|c|c|}
 \hline m & \text{best} & \text{i.i.d. $\mu$} & \text{i.i.d. $\nu_m$} & \revb{ \text{i.i.d. $\nu_m$ (cond.)}} & \text{$\gamma_n^{\nu_m}$} & \text{$\gamma_m^{\otimes n/m}$} & \text{$\gamma_m^{\otimes n/m}$ (cond.)} \\ \hline 10 & 1.3e-01 & 9.0e+01 & 1.5e-01 & 1.5e-01 & 1.5e-01 & 1.4e-01 & 1.4e-01\\ \hline20 & 5.1e-02 & 8.7e+05 & 5.8e-02 & 5.6e-02 & 5.7e-02 & 5.4e-02 & 5.4e-02\\ \hline30 & 2.4e-02 & 1.3e+08 & 2.9e-02 & 2.8e-02 & 2.8e-02 & 2.6e-02 & 2.6e-02\\ \hline40 & 1.3e-02 & 4.6e+07 & 1.6e-02 & 1.5e-02 & 1.5e-02 & 1.4e-02 & 1.4e-02\\ \hline50 & 7.8e-03 & 2.1e+08 & 9.2e-03 & 8.8e-03 & 9.0e-03 & 8.4e-03 & 8.4e-03\\ \hline
  \end{array}
  $$
  \caption{$n=5 m$. Expected relative error $\Ebb(\Vert f - \hat f_m \Vert^2)^{1/2} / \Vert f \Vert$.}
  \end{subtable} \medskip
  
  \begin{subtable}{\textwidth}
$$
\begin{array}{|c|c|c|c|c|c|c|c|}
\hline  m & \text{best} & \text{i.i.d. $\mu$} & \text{i.i.d. $\nu_m$} & \revb{\text{i.i.d. $\nu_m$ (cond.)}} & \text{$\gamma_n^{\nu_m}$} & \text{$\gamma_m^{\otimes n/m}$} & \text{$\gamma_m^{\otimes n/m}$ (cond.)} \\ \hline 10 & 1.32e-01 & 3.4e+02 & 1.9e-01 & 1.7e-01 & 1.8e-01 & 1.6e-01 & 1.6e-01\\ \hline20 & 5.1e-02 & 3.8e+06 & 7.1e-02 & 6.5e-02 & 6.9e-02 & 6.1e-02 & 6.0e-02\\ \hline30 & 2.4e-02 & 4.6e+08 & 3.9e-02 & 3.3e-02 & 3.2e-02 & 3.0e-02 & 2.9e-02\\ \hline40 & 1.3e-02 & 1.9e+08 & 1.9e-02 & 1.8e-02 & 1.8e-02 & 1.6e-02 & 1.5e-02\\ \hline50 & 7.8e-03 & 9.7e+08 & 1.1e-02 & 1.0e-02 & 1.1e-02 & 9.2e-03 & 9.2e-03\\ \hline
  \end{array}
  $$
  \caption{$n=5 m$. Quantile of $\Vert f - \hat f_m \Vert / \Vert f \Vert$ of level $95\%$.}
  \end{subtable}

  \caption{
  For $V_m$ the space of polynomials  of degree $m-1$ and $\mu$ the standard gaussian measure on $\Rbb$, we indicate the expected relative error 
  or the quantile of the relative error of level  95\%,  using $n = r m $ samples from $\nu_m^{\otimes n}$,  the volume-rescaled distribution $\gamma_n^{\nu_m}$,  repeated DPP $\gamma_m^{\otimes r}$, or  repeated DPP conditioned to $S_{\delta}$ with $\delta = 3/4$. The column ``best'' indicates the best approximation error in $V_m.$
 }\label{tab:L2error}
\end{table}

The above numerical experiments illustrate the superiority of repeated DPP distribution $\gamma_m^{\otimes r}$ over i.i.d. optimal sampling, but also over volume-rescaled sampling,  in the small oversampling regime. 
For large oversampling, the different sampling strategies yield similar results. The interest of repeated DPP distribution is that for very small oversampling, stability $S_\delta$ can be achieved with a reasonable probability. This allows for sampling conditioned to $S_\delta$, which  further improves the quality of the least-squares projection. 

\revb{We observe in Table \ref{tab:L2error} that i.i.d. sampling with conditioning yields almost 
the same performance as repeated DPP with conditioning. This proves the interest of conditioning.  However, we were not able to generate an i.i.d. sample of size  $n=2m$ in reasonable time for $m\ge 30$ (crosses in Table 1). This is explained by the fact that using i.i.d. sampling requires a high number of samples to satisfy $S_\delta$ with a reasonable probability. 
This proves the advantage of using repeated DPP, which allows to use a conditioning technique with a very small sample size (even $2m$). }

%
%

\section*{Acknowledgements}
This project is funded by the ANR-DFG project COFNET (ANR-21-CE46-0015).
This work was partially conducted within the France 2030 framework programme, Centre Henri Lebesgue ANR-11-LABX-0020-01.

\clearpage 

\appendix


\section{Some known results on random matrices}

\begin{theorem}[Matrix Chernoff inequality \cite{tropp2012user}]\label{th:matrix-chernoff}
Let $\A_1,\hdots \A_n$ be independent random symmetric matrices of size $m$-by-$m$ such that for all $i$, $0 \le \lambda_{min}(\A_i) $ and $\lambda_{max}(\A_i) \le L$. 
Then  
\begin{align}
&\mathbb{P}(\lambda_{max}(   \sum_{i=1}^n \A_i )\ge (1+\delta) \mu_{max}) \le m \exp(- d_\delta \mu_{max} / L) \quad \text{for } \epsilon \ge 0, \quad \text{and}\\
&\mathbb{P}(\lambda_{min}(  \sum_{i=1}^n \A_i) \le (1-\delta) \mu_{min}) \le m \exp(- c_\delta \mu_{min} / L) \quad \text{for } \epsilon\in [0,1),
\end{align}
where $\mu_{min}= \lambda_{min}(\mathbb{E}(\sum_{i=1}^n \A_i ))$,  $\mu_{max}= \lambda_{max}(\mathbb{E}(\sum_{i=1}^n \A_i ))$, $c_\delta = \delta + (1-\delta) \log(1-\delta)$ and $d_\delta=-\delta + (1+ \delta)\log(1+\delta).$ It holds $\frac{5}{13} \delta^2 \le d_\delta  \le \delta^2/2 \le  c_\delta \le  \delta^2$.
\end{theorem}
\begin{proof}
We provide a sketch of the proof of Tropp \cite{tropp2015introduction} for the bound on the minimal eigenvalue. Let $\B =  \sum_{i=1}^n \A_i$. For any $\theta <0$, it holds 
$$
\Pbb(\lambda_{min}(\B) \le t  ) = \Pbb(e^{ \theta \lambda_{min}(\B) } \ge e^{ \theta t}   ) =  \Pbb(e^{\lambda_{min}(\theta \B) } \ge e^{\theta t}   ) \le e^{-\theta t} \Ebb(e^{\lambda_{min}(\theta \B) } ),
$$
where the last inequality is given by Markov inequality. Then using $e^{\lambda_{min}(\theta \B) } = \lambda_{min}(e^{\theta \B}) \le \trace(e^{\theta\B} )$, we obtain 
$$
\Pbb(\lambda_{min}(\B) \le t  ) \le \inf_{\theta<0} e^{-\theta t} \Ebb( \trace e^{\theta \B } ).
$$
For positive-definite matrix $\H$, the map $\A \mapsto \trace e^{\H + \log(\A)}$ is concave on the positive cone of positive definite matrices. Letting $\X_i :=  e^{\theta \A_i}$, a sequential application of Jensen's inequality gives
$$\Ebb( \trace e^{\theta \B } ) = \Ebb( \trace \exp ( \sum_{i=1}^n \log(\X_i) )) \le  \trace \exp( \sum_{i=1}^n \log(\Ebb(\X_i)) ) =  \trace \exp( \sum_{i=1}^n \log(\Ebb( e^{\theta \A_i})))  . $$
Also it holds $\log \Ebb(e^{\theta \A_i}) \preceq g(\theta) \Ebb(\A_i)$ where $g(\theta) = L^{-1}(e^{\theta L}-1),$ so that
$$\trace \exp( \sum_{i=1}^n \log(\Ebb( e^{\theta \A_i}))) \le \trace \exp( g(\theta) \Ebb(\sum_{i=1}^n \A_i)) \le n e^{g(\theta) \mu_{min}}  .$$ Therefore, 
$$\Pbb(\lambda_{min}(\B) \le t  ) \le \inf_{\theta<0}  n e^{-\theta t} e^{g(\theta) \mu_{min}}  $$
Taking $t  =  (1-\delta) \mu_{min}$, the infimum is attained at $\theta = L^{-1} \log(1-\delta)$, which gives the desired result. 
\end{proof}
\begin{lemma}[Lemma 2.3 in \cite{derezinski2022unbiased}]\label{lem:expectation_det}
Let $\A,\B \in \Rbb^{n\times m}$ be two random matrices whose row vectors are drawn as an i.i.d. sequence of $n$ pairs of random vectors $(\a_i , \b_i)$. Then 
 \begin{align*}
n^m \Ebb(\det(\A^T\B)) =& \frac{n!}{(n-m)!} \det(\Ebb(\A^T\B))\quad \text{for any } n\ge m, \; \text{and} \\
n^{m-1} \Ebb(\adj(\A^T\B)) =& \frac{n!}{(n-m+1)!} \adj(\Ebb(\A^T\B))\quad \text{for any } n\ge m-1.
 \end{align*}
\end{lemma}
\begin{lemma}[Lemma 2.11 in \cite{derezinski2022unbiased}]\label{lem:formula_det}
For any matrix $\A  \in \Rbb^{n\times m}$ with $n>m,$ 
it holds  
 \begin{align*}
\det(\A^T\A) \A^\dagger  = \frac{n}{n-m} \sum_{i=1}^n \det(\A^T (\I - \e_i\e_i^T) \A) ( (\I - \e_i\e_i^T) \A)^\dagger.
\end{align*}
\end{lemma}


\section{Properties of projection determinental point processes}\label{app:dpp}

\begin{ProofOf}{Proposition \ref{prop:marginaldpp}}
This is a standard result on projection determinantal processes, see e.g. \cite{lavancier2015dpp}. 
We here provide a short proof with our notations. 
 The fact that all marginals are the same comes from the invariance to permutations of the distribution $\gamma_m$. From the classical ``base times height formula'' for a determinant, we have
$$
\det(\Phi(\x))  = \det((\varphib(x_1) , \hdots \varphib(x_m)))= \Vert \varphib(x_1) \Vert_2 \Vert P_{W_{1}^\perp} \varphib(x_2) \Vert_2  \hdots \Vert P_{W_{m-1}^\perp} \varphib(x_m) \Vert_2  
$$  
where  $P_{W_k^\perp} = I_m - P_{W_k}$ is the orthogonal projection onto the orthogonal complement of $W_k = \mathrm{span}\{\varphib(x_1) , \hdots , \varphib(x_k)\} $ in $\Rbb^m$. Therefore, the density of $\gamma_m$ with respect to $\mu^{\otimes m}$ has the following expression 
$$
\frac{1}{m!}\det(\Phi(\x)^T \Phi(\x)) = \prod_{k=1}^m p_k(x_k), \quad p_k(x_k) := \frac{1}{m-k+1} \Vert   P_{W_{k-1}^\perp} \varphib(x_k) \Vert^2_2,
$$  
with the convention $W_0 = \{0\}.$ The function $p_k$ depends on $(x_1,\hdots ,x_{k-1})$ and is a \revb{probability} density since 
\begin{align*}
\int p_k(x) d\mu(x) &=   \frac{1}{m-k+1} \int \varphib(x)^T P_{W_{k-1}^\perp} \varphib(x) d\mu(x) 
\\
&=  \frac{1}{m-k+1}  \trace  (P_{W_{k-1}^\perp}  \int  \varphib(x)  \varphib(x)^Td\mu(x) ) \\
&=  \frac{1}{m-k+1}  \trace  (P_{W_{k-1}^\perp} ) = 1,
\end{align*}
where we have used the fact that $\trace  (P_{W_{k-1}^\perp} ) = \dim(W_{k-1}^\perp) = m - k +1. $
That provides a factorization of $\gamma_m$ in terms of the marginal $p_1(x_1) d\mu(x_1) =  \frac{1}{m} \Vert \varphib(x_1) \Vert_2^2 d\mu$  and the conditional  distributions $p_k(x_k) d\mu(x_k)$ of $ x_{k}$ knowing $(x_1,\hdots,x_{k-1}),$ which ends the proof. 

\end{ProofOf}

The next result provides the distribution of all marginal distributions of $\gamma_m$. \reva{This is a  standard result on DPPs (see, e.g., \cite[Lemma 3.3]{derezinski2022unbiased}). For completeness, we here provide a proof with our notations}.  
  
\begin{proposition}\label{prop:Smarginaldpp}
Let $\x  = (x_1, \hdots , x_m)\sim \gamma_m$.
For a nonempty tuple $T \subset [m]$,   $\x_T  = (x_i)_{i\in T}$ has for distribution 
$$
\frac{(m-|T|)!}{m!}\det(\Phi(\y)\Phi(\y)^T) d\mu^{\otimes |T| }(\y), \quad \y\in \Xc^{|T|} .
$$
\end{proposition}
\begin{proof}
Because of the symmetry of the distribution $\gamma_m$, it is sufficient to consider sets $T =[k]$ and $\x_T = \x_{[k]} = (x_1,\hdots,x_k)$. Let denote $p_{[k]} \mu^{\otimes k}$ the distribution of $\x_{[k]}$.  
The result is true for $k=1$. Then we proceed by induction. From Proposition \ref{prop:marginaldpp}, we know that the conditional distribution of $x_{k+1}$ knowing $\x_{[k]}$ is 
\begin{align*}
\frac{1}{m-k} (\Vert \varphib(x_{k+1}) \Vert^2_2 - \Vert  \Phi(\x_{[k]})(\Phi(\x_{[k]})\Phi(\x_{[k]})^T)^{-1} \Phi(\x_{[k]})^T   \varphib(x_{k+1}) \Vert_2^2  ) d\mu.
\end{align*}
Assuming the distribution of $\x_{[k]}$ is $$\frac{(m-k)!}{m!} \det(\Phi(\x_{[k]})\Phi(\x_{[k]})^T) d\mu^{\otimes k}(\x_{[k]}),$$ we deduce that the distribution of $\x_{[k+1]} = (\x_{[k]} , x_{k+1})$ admits as density with respect to $\mu^{\otimes(k+1)}$ 
\begin{align*}
&\frac{(m-k-1)!}{m!}\det(\Phi(\x_{[k]})\Phi(\x_{[k]})^T)  (\Vert \varphib(x_{k+1}) \Vert^2_2 \\
  &\quad- \Vert  \Phi(\x_{[k]})(\Phi(\x_{[k]})\Phi(\x_{[k]})^T)^{-1} \Phi(\x_{[k]})^T   \varphib(x_{k+1}) \Vert_2^2  )\\
 &=\frac{(m-k-1)!}{m!}\det(\Phi(\x_{[k]})\Phi(\x_{[k]})^T)  (  \varphib(x_{k+1})^T \varphib(x_{k+1})  \\
   &\quad  - \varphib(x_{k+1})^T   \Phi(\x_{[k]})(\Phi(\x_{[k]})\Phi(\x_{[k]})^T)^{-1} \Phi(\x_{[k]})^T   \varphib(x_{k+1})    )
\\ &=\frac{(m-k-1)!}{m!} \det(\Phi(\x_{[k+1]})\Phi(\x_{[k+1]}^T)),
\end{align*}
which ends the proof.

\end{proof}

\section{Properties of volume sampling}\label{app:vs}

We first provide a straightforward result. 
\begin{lemma}
Assume $\x = (x_1,\hdots,x_n)$ is drawn from the distribution $\gamma_n^\nu$ with $\nu = \revb{w^{-1}}\mu$ \revb{a probability measure}.
For any measurable function $g : \Xc^n \to \Rbb,$
$$
\Ebb_{\x\sim \gamma_n^\nu}(g(\x)) = \Ebb_{\y\sim \nu^{\otimes n}}(\frac{(n-m)!}{n!} \det(\Phi^w(\y)^T \Phi^w(\y)) g(\y)).
$$
\end{lemma}

\revb{We next provide a generalization of \cite[Theorem 2.10]{derezinski2022unbiased} which is fundamental to prove the unbiasedness of   weighted least-squares projections based on  volume sampling  with general reference measures. }
\begin{lemma}\label{lem:vs-unbiased-pseudoinverse}
Assume $\x = (x_1,\hdots,x_n)$ is drawn from the distribution $\gamma_n^\nu$ with $\nu = \revb{w^{-1}}\mu$ \revb{a probability measure}. Then for any function $f$,  it holds
$$
\Ebb(\Phi^w (\x)^{\dagger}f^w(\x))  =  \revb{\int \varphib(y) f(y) d\mu(y)},
$$
where $f^w = f \revb{w^{1/2}}.$
\end{lemma}
\begin{proof}
First consider the case $n=m$, where $\x = (x_1,\hdots,x_m)$ is drawn from $\gamma_m = \DPP_\mu(V_m).$ 
In this case, $\Phi(\x)$ is a square matrix, almost surely invertible, and 
$ \Phi^w (\x)^{\dagger}f^w(\x) =\Phi(\x)^{\dagger}f(\x) $. 
We obtain using Cramer's rule\footnote{For any matrix $\A\in \Rbb^{n\times n}$ and vector $\b \in \Rbb^n$, $\det(\A) (\A^\dagger \b)_i= \det(\A + (\b - \A\e_i)\e_i^T)$}
\revb{
\begin{align*}
\Ebb(\Phi(\x)^{\dagger}f(\x))_i
 &=\frac{1}{m!} \int \left(\det(\Phi(\y)^T\Phi(\y))\Phi(\y)^{\dagger}f(\y) \right)_i d \mu(y) \\
&=\frac{1}{m!} \int \det(\Phi(\y)^T) \det ( \Phi(\y) + (f(\y) - \Phi(\y)\e_i)\e_i^T ) d \mu(y)\\
&= \frac{1}{m!} \int  \det(\Phi(\y)^T( \Phi(\y) + (f(\y) - \Phi(\y)\e_i)\e_i^T ) ) d \mu(y)\\
&\overset{(*)}{=}  \det \left( \frac{1}{m}\int \Phi(\y)^T( \Phi(\y) + (f(\y) - \Phi(\y)\e_i)\e_i^T ) d \mu(y) \right)\\
&=\det \left(\I + ( \int \varphib(y) f(y) d \mu(y) - \e_i  ) \e_i^T \right) = \left( \int \varphib(y) f(y) d \mu(y)  \right)_i,
\end{align*}
}
where $(*)$ is deduced from Lemma \ref{lem:expectation_det}. For the case $n>m$, we proceed by induction. Letting $\y\sim \nu^{\otimes n}$ and using  Lemma \ref{lem:formula_det}, we have 
\begin{align*}
\Ebb(\Phi^w(\x)^{\dagger} f^w(\x)) &= \frac{(n-m)!}{n!} \Ebb(\det(\Phi^w(\y)^T\Phi^w(\y))\Phi^w(\y)^{\dagger}f^w(\y))\\
&= \frac{(n-m)!}{n!} \frac{n}{n-m}  \sum_{i=1}^n \Ebb( \det(\Phi^w(\y)^T (\I - \e_i\e_i^T) \Phi^w(\y)) ( (\I - \e_i\e_i^T) \Phi^w(\y))^\dagger f^w(\y))
\\
&= \frac{(n-m-1)!}{(n-1)!}   \sum_{i=1}^n \Ebb( \det(\Phi^w(\y_{- i})^T   \Phi^w(\y_{- i})) (  \Phi^w(\y_{- i}))^\dagger f^w(\y_{- i}))\\
&=\Ebb( \Phi^w(\tilde \x)^{-1} f^w(\tilde \x)),
\end{align*}
with $\tilde \x \sim \gamma_{n-1}^\nu,$ and $\y_{-i}$ is the vector $\y$ without the $i$-th component. We then deduce $\Ebb(\Phi^w(\x)^{\dagger} f^w(\x)) = \revb{ \int \varphib(y) f(y) d\mu(y)} $.\\
\end{proof}

\begin{lemma}\label{lem:vs-variance_quasi_projection}
Assume $ (x_1,\hdots,x_n)$ is drawn from the distribution $\gamma_n^\nu$ with $\nu = \revb{w^{-1}}\mu$ and $\revb{w^{-1}}\ge \alpha \revb{w_m^{-1}}$. Then for any function $f$ and $f^w=f \revb{w^{1/2}}$,  it holds 
\begin{align*}
\Ebb( \Vert \frac{1}{n} \Phi^w(\x)^T f^w(\x)   \Vert^2 ) &\le  \frac{m}{n} \alpha^{-1} ( \beta + \xi m \alpha^{-1})   \Vert f \Vert^2  + \Vert P_{V_m} f \Vert^2
\end{align*}
with $\beta =1 + (\alpha^{-1}-1) \frac{m}{n}$ and 
 $\xi  = 0$ if $\nu = \nu_m$ or $\xi = 1$  if $\nu\neq \nu_m$.
In the case where $\varphib(x_i) f(x_i) \ge 0$ almost surely, it holds 
\begin{align*}
\Ebb(\Vert \frac{1}{n} \Phi^w(\x)^T f^w(\x)  \Vert^2 ) &\le  \frac{m}{n} \alpha^{-1} \beta   \Vert f \Vert^2  + (1+ 2 \alpha^{-2} \frac{m}{n}) \Vert P_{V_m} f \Vert^2
\end{align*}
\end{lemma}
\begin{proof}
Let $\b = \frac{1}{n}\Phi^w(\x)^T f^w(\x) = \frac{1}{n} \sum_{k=1}^n \varphib(x_k)^T f(x_k) \revb{w(x_k)}. $ Letting $\a(x) := \varphib(x) f(x) \revb{ w(x)}$, we have 
\begin{align*}
 \Ebb(\Vert \b \Vert_2^2) &= \frac{1}{n^2}   \sum_{k,l=1}^n  \Ebb((\a(x_k),\a(x_l) ) )    
 = \frac{1}{n^2} \sum_{k=1}^n  \Ebb( \Vert \a(x_k) \Vert^2  )   +    \frac{1}{n^2}  \sum_{\substack{k,l=1\\ k\neq l}}^n  \Ebb((\a(x_k),\a(x_l) ) ) 
 \end{align*}  
 The marginal distribution of 
$\gamma_n^\nu$ is $\tilde \nu = \revb{\tilde w^{-1}} \mu$ with $\revb{\tilde w^{-1}} = \frac{n-m}{n} \revb{  w^{-1}} + \frac{m}{n} \revb{  w_m^{-1}}$ such that 
$
\revb{\tilde w^{-1}}  
\le \beta \revb{  w^{-1}} $ with $ \beta =1 + (\alpha^{-1}-1) \frac{m}{n}$.  Then, 
 \begin{align*}
\frac{1}{n^2} \sum_{k=1}^n  \Ebb( \Vert \a(x_k) \Vert^2  ) &=\frac{1}{n}  \Ebb_{x\sim \tilde \nu}( \Vert \varphib(x)\Vert_2^2 f(x)^2 \revb{w(x)^{2}} )\\
& \le  \frac{m}{n} \alpha^{-1} \beta \, \Ebb_{x\sim \tilde \nu}(f(x)^2 \revb{ \tilde w(x)}) \\
&= \frac{m}{n} \alpha^{-1} \beta \Vert f \Vert^2.
 \end{align*}
 Up to a permutation, we can now consider that $(x_1,\hdots,x_m) \sim \gamma_m$ and $(x_{m+1}, \hdots,x_n)\sim \nu^{\otimes (n-m)}$ are independent.   Letting $z\sim \nu$, we have 
 \begin{align*}
 &\Ebb(  \sum_{\substack{k,l=1\\ k\neq l}}^n  (\a(x_k),\a(x_l) ) )
 = m(m-1) \Ebb(   (\a(x_1),\a(x_2) )    ) \\
 &\quad + 2 m(n-m)  ( \Ebb(\a(x_1)),\Ebb(\a(z) ))   + (n-m)(n-m-1) \Vert \Ebb(\a(z)) \Vert^2_2 
 \end{align*}
We have  $$\Vert \Ebb(\a(z)) \Vert_2  = \Vert \Ebb(\varphib(z)f(z) \revb{w(z)})  \Vert_2 = \Vert \revb{ \int \varphib(y) f(y) d\mu(y)} \Vert_2^2 = \Vert P_{V_m} f \Vert_2.$$
Letting $(y_1,y_2)\sim \mu^{\otimes 2}$, and using Proposition \ref{prop:Smarginaldpp}, we obtain 
\begin{align*}
   m(m-1) \Ebb(   (\a(x_1),\a(x_2) )    ) &=   \revb{\int} (\a(y_1),\a(y_2) )  \det(\Phi(y_1,y_2)\Phi(y_1,y_2)^T ) \revb{d\mu(y_1) d\mu(y_2)}\\
    &=  \revb{\int} (\a(y_1),\a(y_2) ) (\Vert \varphib(y_1) \Vert^2\Vert \varphib(y_2) \Vert^2 - (\varphib(y_1),
    \varphib(y_2))^2 ) \revb{d\mu(y_1) d\mu(y_2)}
\\
&\le    \Vert \revb{\int} \a(y_1) \Vert \varphib(y_1) \Vert^2 \revb{ d\mu(y_1)} \Vert_2^2  \\
&\quad  -  \revb{\int} (\varphib(y_1),
    \varphib(y_2))^3 f(y_1)f(y_2) \revb{w(y_1) w(y_2) }  \revb{d\mu(y_1) d\mu(y_2)}
   \end{align*}
   The second term in the above upper bound can be written $$\revb{\int} \sum_{l} g_l(y_1) g_l(y_2) \revb{d\mu(y_1) d\mu(y_2)} = \sum_{l}  \revb{\int} g_l(y))^2 \revb{d\mu(y)} $$ for some functions $g_l$, so that 
    \begin{align*}
   m(m-1) ( \Ebb(\a(x_1)),\Ebb(\a(x_2) ))& \le  m^2 \Vert \Ebb( \a(x)) \Vert_2^2,
   \end{align*}
   with $x\sim \nu_m$. 
Gathering the above results, we get 
 \begin{align*}
   \Ebb(\Vert \b \Vert_2^2) \le &\frac{m}{n} \alpha^{-1} \beta \Vert f \Vert^2 + \frac{m^2}{n^2} \Vert \Ebb( \a(x)) \Vert_2^2 \\
   &+   
  \frac{2m(n-m)}{n^2} \Vert \Ebb( \a(x)) \Vert_2 \Vert P_{V_m} f \Vert + \frac{(n-m)(n-m-1)}{n^2} \Vert P_{V_m} f\Vert^2.
   \end{align*}
   If   $w = w_m$, then $\alpha = \beta =1$ and $\Vert \Ebb(\a(x)) \Vert_2 = \Vert P_{V_m} f\Vert$, and therefore 
 \begin{align*}
     \Ebb(\Vert \b \Vert_2^2) &\le \frac{m}{n}  \Vert f \Vert^2 + \frac{m^2 +2m(n-m) +  (n-m)(n-m-1)}{n^2} \Vert P_{V_m} f\Vert^2\\
     &\le \frac{m}{n}  \Vert f \Vert^2 + \frac{  n^2-n + m}{n^2} \Vert P_{V_m} f\Vert^2 \le \frac{m}{n}  \Vert f \Vert^2 + (1- \frac{n-m}{n^2}) \Vert P_{V_m} f\Vert^2 
   \end{align*}
 If $\revb{w^{-1} } \ge \alpha \revb{w_m^{-1} }$, we have  
 $$
 \Vert \Ebb(\a(x)) \Vert_2 \le \Ebb( \Vert \a(x) \Vert_2^2 )^{1/2} =  m^{1/2} (\revb{\int}  f(y)^2 \revb{w(y)^{2}  w_m(y)^{-2}} \revb{d\mu(y)} )^{1/2} \le m^{1/2}\alpha^{-1} \Vert f \Vert,
  $$
   and therefore, letting $\xi_m = m^{1/2}$, 
   \begin{align*}
   \Ebb(\Vert \b \Vert_2^2) &\le \frac{m}{n} \alpha^{-1} \beta \Vert f \Vert^2 + \frac{m^2}{n^2} \xi_m^2 \alpha^{-2} \Vert f \Vert^2 \\
  &\quad+   
  \frac{2m(n-m)}{n^2} \xi_m \alpha^{-1} \Vert f \Vert \Vert P_{V_m} f \Vert + \frac{(n-m)(n-m-1)}{n^2} \Vert P_{V_m} f\Vert^2
  \\
  &\le (\frac{m}{n} \alpha^{-1} \beta + \frac{m}{n} \xi_m^2 \alpha^{-2} )  \Vert f \Vert^2 + \frac{(n-m)(n-m-1) + m(n-m) }{n^2}  \Vert P_{V_m} f \Vert^2 \\
  &\le \frac{m}{n} \alpha^{-1} ( \beta + \xi_m^2 \alpha^{-1})   \Vert f \Vert^2  + \Vert P_{V_m} f \Vert^2
     \end{align*}
   If the particular case where $\varphib(y) f(y) \ge 0$ almost surely, then 
   $$
   \Vert \Ebb(\a(x)) \Vert_2  = \Vert  \revb{\int} \varphib(y) f(y) \revb{w(y) w_m(y)^{-1}}\revb{d\mu(y)} \Vert_2 \le  \alpha^{-1} \Vert \revb{\int}\varphib(y) f(y) \revb{d\mu(y)} \Vert_2  = \alpha^{-1} \Vert P_{V_m} f\Vert,
   $$
   and we get 
   \begin{align*}
   \Ebb(\Vert \b \Vert_2^2) &\le \frac{m}{n} \alpha^{-1} \beta \Vert f \Vert^2 + ( \frac{m^2}{n^2} \alpha^{-2} +  \frac{ 2m(n-m)}{n^2} \alpha^{-1} + \frac{(n-m)(n-m-1)}{n^2} ) \Vert P_{V_m}f \Vert^2\\
   &\le  \frac{m}{n} \alpha^{-1} \beta \Vert f \Vert^2 +  (1 +  2\frac{m}{n} \alpha^{-2} )\Vert P_{V_m} f\Vert^2
     \end{align*}   
\end{proof}

\begin{ProofOf}{Lemma \ref{lem:vs-invG}}
The first statement results from \cite[Theorem 2.9]{derezinski2022unbiased}. We here provide the proof for completeness.
First note that $\G^w(\x)$ is invertible almost surely. 
Letting $\y \sim \nu^{\otimes n}$, we have 
\begin{align*}
\Ebb(\G^w(\x)^{-1}) &=  \Ebb(\G^w(\x)^{\dagger}) = n^m \frac{(n-m)!}{n!}\Ebb(\G^w(\y)^{\dagger} \det(\G^w(\y))  ) 
\\
&\preceq  n^m \frac{(n-m)!}{n!} \Ebb(\adj(\G^w(\y)) )  \\
&\overset{(*)}{=}  n^m \frac{(n-m)!}{n!} \frac{n!}{(n-m+1)! n^{m-1}} \adj(\Ebb(\G^w(\y))  )  \\
&= \frac{n}{n-m+1} \, \adj( \I )  = \frac{n}{n-m+1} \,  \I  ,
\end{align*}
where $(*)$ is obtained from Lemma \ref{lem:expectation_det}, and where the Loewner ordering $\preceq$ can be replaced by an equality when $\Phi^w(\y)$ has rank $m$ almost surely, which implies that $\G^w(\y)^{\dagger} = \G^w(\y)^{-1}$.
We deduce from the first statement that 
   $\G := \G^w(\x)$ satisfies 
$
\Ebb(\lambda_{min}(\G)^{-1})  = \Ebb(\lambda_{max}(\G^{-1})) \le  
\Ebb(\trace(\G^{-1})) = \trace(\Ebb(\G^{-1})) \le \frac{nm}{n-m+1}.
$
\end{ProofOf}

\begin{ProofOf}{Lemma \ref{lem:vs-lambdamin}}
The distribution of $\G^w$ is the same as the Gram matrix associated with $m$ samples from $\gamma_m$ and $n-m$ i.i.d.  samples from $\nu$, independent from the first $m$ samples. Then write 
$\G^w = \frac{m}{n} {\G_{[m]}^w} + \frac{n-m}{n}  \G^w_{[m]^c}$, where $ \G^{ w}_{[m]^c}$ is the Gram matrix associated with $n-m$ i.i.d.  samples from $ \nu$, and  $\G^w_{[m]}$ is the Gram matrix associated with $m$ points from  the distribution $\gamma_m$.  Matrices $\G^w_{[m]}$ and $\G^w_{[m]^c}$ are independent. 
First note that for $\A$ and $\B$ symmetric positive definite, $\lambda_{min}(\A+\B)^{-1}  \le  \lambda_{min}(\A)^{-1}$. We then deduce that 
$$\lambda_{min}(\G^w)^{-1} \le  \frac{n}{m} \lambda_{min}(\G^w_{[m]})^{-1} \quad \text{and} \quad 
 \lambda_{min}(\G^w)^{-1} \le    \frac{n}{n-m}  \lambda_{min}( \G^w_{[m]^c})^{-1}.  
 $$
Noting that $K_{m,w} \le \alpha^{-1} m$, we have  from Lemma \ref{lem:matrix-chernoff}  that  the event 
$S = \{\lambda_{min}({ \G^w_{[m]^c}}) < 1-\delta\}$ satisfies  
$\Pbb( S) \le m \exp(- \frac{c_\delta (n-m) \alpha }{m} ):=\eta(n-m,m)  $. 
We  deduce that 
$$
\Pbb(\lambda_{min}(\G^w)^{-1} > (1-\delta)^{-1} \frac{n}{n-m}) \le \Pbb(\lambda_{min}( \G^w_{[m]^c})^{-1} > (1-\delta)^{-1})\le \eta(n-m,m),
$$
that is the second statement. 
%
For the final statement, we have that  
\begin{align*}
\Ebb(\lambda_{min}(\G^w)^{-1}) &\le \Ebb(\lambda_{min}(\G^w)^{-1} \vert S) \eta(n-m,m) + \Ebb(\lambda_{min}(\G^w)^{-1} \vert S^c)  \\
&\le \frac{n}{m}  \Ebb(\lambda_{min}(\G^w_{[m]})^{-1}) \eta(n-m,m) + \frac{n}{n-m} \Ebb(\lambda_{min}( \G^w_{[m]^c})^{-1} \vert S^c)\\
&\le   nm \eta(n-m,m) + \frac{n}{n-m}  (1-\delta)^{-1}
\end{align*}
where we have used the independence of $\G^w_{[m]}$ and $S$, and the second statement of Lemma \ref{lem:vs-invG} applied to $\G^w_{[m]}$. Then it holds 
 \begin{align*}
\Ebb(\lambda_{min}(\G^w)^{-1}) &\le nm^2  \exp(- \frac{c_\delta (n-m) \alpha }{m} ) + \frac{n}{n-m} (1-\delta)^{-1}
 \end{align*}
which concludes the proof.
\end{ProofOf}

\begin{lemma}\label{lem:vs-variance_projection}
Let $\x \sim \gamma^\nu_n$ with $\nu = \revb{w^{-1}} \mu$ and $\revb{w^{-1}} \ge \alpha \revb{w_m^{-1}} .$ 
 For an arbitrary function $g$, provided  $n\ge 2m+2$ and $n \ge 2 m \alpha^{-1} c_{\delta}^{-1} \log(\zeta^{-1} m^2 n )$,
it holds 
 $$\Ebb(\Vert \Phi^w(\x)^\dagger g^w(\x) \Vert_2^2) \le (4\frac{m}{n}(1-\delta)^{-2}(\beta + \xi m \alpha^{-1}) + \alpha^{-1} \zeta  ) \Vert g \Vert^2 +   4 (1-\delta)^{-2} \Vert P_{V_m}{g} \Vert^2 $$
with  $\beta =1 + (\alpha^{-1}-1) \frac{m}{n}$,  and  $\xi = 0$ if $\nu=\nu_m$ or $\xi = 1$ if $\nu\neq \nu_m.$
 \end{lemma}
\begin{proof}
 First note that   $ \Phi^w(\x)^\dagger g^w(\x)  = \G^w(\x)^{-1} \b(\x)$ with $\b(\x) = \frac{1}{n} \Phi^w(\x)^T g^w(\x).$ 
Up to a reordering, assume that $(x_1,\hdots,x_n) \sim \gamma_m \otimes \nu^{\otimes n-m}$. 
Let $ m \le s   \le n $. Then write 
$\G^w(\x) := \frac{s}{n} {\G^w_{[s]}} + \frac{n-s}{n}  \G^w_{[s]^c}$, where $ \G^w_{[s]^c}$ is the Gram matrix associated with $n-s$ i.i.d.  samples from $ \nu$, and  $\G^w_{[s]}$ is the Gram matrix associated with $s$ points from  the distribution $\gamma_s^\nu$.  Matrices $\G^w_{[s]}$ and $ \G^w_{[s]^c}$ are independent. 
Let $A = \{\lambda_{min}(\G^w_{[s]^c}) \ge (1-\delta)\frac{n-s-1}{n-s}\}$. 
We have 
\begin{align*}
\Ebb(\Vert \Phi^w(\x)^\dagger g^w(\x) \Vert_2^2) = \Ebb(\Vert \Phi^w(\x)^\dagger g^w(\x) \Vert_2^2 | A) \Pbb(A) + \Ebb(\Vert \Phi^w(\x)^\dagger g^w(\x) \Vert_2^2 | A^c) \Pbb(A^c) 
\end{align*}
For the first term, we have 
\begin{align*}
\Ebb(\Vert \Phi^w(\x)^\dagger g^w(\x) \Vert_2^2 | A) \Pbb(A)
&\le  \Ebb( \Vert \G^w(\x)^{-1} \Vert_2^{2} \Vert  \b \Vert_2^2 | A) \Pbb(A)\\
&\le \frac{n^2}{(n-s)^2} \Ebb( \lambda_{min}(\G^w_{[s]^c}(\x))^{-2} \Vert \b \Vert_2^2 | A) \Pbb(A)\\
&\le \frac{n^2}{(n-s-1)^2}(1-\delta)^{-2} \Ebb( \Vert \b \Vert_2^2),
\end{align*}
and using Lemma \ref{lem:vs-variance_quasi_projection}, we obtain 
\begin{align*}
\Ebb(\Vert \Phi^w(\x)^\dagger g^w(\x) \Vert_2^2 | A) \Pbb(A) \le \frac{n^2}{(n-s-1)^2}(1-\delta)^{-2}(\frac{m}{n} \alpha^{-1} (\beta + \xi m \alpha^{-1}) \Vert g \Vert^2 + \Vert P_{V_m} g \Vert^2)
\end{align*}
with $\beta = 1 + (\alpha^{-1} -1) \frac{m}{n} $, and $\xi = 0$ if $\nu=\nu_m$ or $\xi = 1$ if $\nu\neq \nu_m.$ \\
For the second term, noting that  $\Vert  \Phi^w(\x)^\dagger \Vert_2^2 = \Vert (\Phi^w(\x)^T\Phi^w(\x))^{-1} \Vert_2 = n^{-1} \Vert \G^w(\x)^{-1} \Vert_2 \le s^{-1} \lambda_{min}( \G^w_{[s]})^{-1}$, 
we have 
\begin{align*}
&\Ebb(\Vert \Phi^w(\x)^\dagger g^w(\x) \Vert_2^2 | A^c) \le s^{-1} \Ebb(\lambda_{min}( \G^w_{[s]})^{-1} \Vert g^w(\x) \Vert^2_2 | A^c ) \\
&= s^{-1} \Ebb( \lambda_{min}( \G^w_{[s]})^{-1} \Vert g^w(\x_{[s]}) \Vert^2_2 ) +  s^{-1}\Ebb( \lambda_{min}( \G^w_{[s]})^{-1}) \Ebb( \Vert g^w(\x_{[s]^c}) \Vert^2_2 | A^c ) \\
&\le   \Ebb( \lambda_{min}( \G^w_{[s]})^{-1} g^w(x_1)^2 ) 
+ \frac{m}{s-m+1}\Ebb( \Vert g^w(\x_{[s]^c}) \Vert^2_2 | A^c ) ,
\end{align*}
where we have the invariance through permutations of $\gamma_s$ and the independence of $\G^w_{[s]}$ and $g^w(\x_{[s]^c})$.
Letting $B =  \{\lambda_{\min}(\G^w_{[s+1 , n-1 ]^c}) \ge (1-\delta)\} \subset A$, we have 
\begin{align*}
\Ebb( \Vert g^w(\x_{[s]^c}) \Vert^2_2 | A^c ) &= \sum_{i=s+1}^n \Ebb( g^w(x_i)^2  \vert A^c) =  (n-s) \Ebb( g^w(x_n)^2  \vert A^c)\\
&\le (n-s) \Ebb( g^w(x_n)^2  \vert B^c) \frac{\Pbb(B^c)}{\Pbb(A^c)} = (n-s) \Ebb( g^w(x_n)^2)  \frac{\Pbb(B^c)}{\Pbb(A^c)} 
\end{align*}
so that 
\begin{align*}
\Ebb( \Vert g^w(\x_{[s]^c}) \Vert^2_2 | A^c )\Pbb(A^c)&\le (n-s) \Vert g \Vert^2  m \exp(-\frac{c_{\delta}  \alpha (n-s-1) }{m})
\end{align*} 
Finally, using Lemma \ref{lem:dpp_exp_tr}, we obtain 
\begin{align*}
  \Ebb( \lambda_{min}( \G^w_{[s]})^{-1} g^w(x_1)^2 )  &\le   \Ebb( \trace(( \Phi^w(\x_{[s]})^T\Phi^w(\x_{[s]}))^{-1}) g^{w}(x_1)^2 ) 
\\
& \le \alpha^{-1} \frac{m}{s-m+1} \Ebb_{x\sim \nu}(g^w(x)^2) \\
& =  \alpha^{-1} \frac{m}{s-m+1} \Vert g \Vert^2.
\end{align*}
Gathering the above results, we have  
\begin{align*}
\Ebb(\Vert \Phi^w(\x)^\dagger g^w(\x) \Vert_2^2) &\le \frac{n^2}{(n-s-1)^2}(1-\delta)^{-2}(\frac{m}{n} \alpha^{-1} (\beta+ \xi m \alpha^{-1}) \Vert g \Vert^2 + \Vert P_{V_m} g \Vert^2) \\
&\quad + \frac{m^2}{s-m+1} (  (n-s) +\alpha^{-1})  \Vert g \Vert^2  \exp(-\frac{c_{\delta}  \alpha (n-s-1) }{m})  \\
&\le C\Vert g\Vert^2 + D\Vert P_{V_m} \Vert^2
\end{align*}
with 
$$
D = \frac{n^2}{(n-s-1)^2}(1-\delta)^{-2} 
$$
and 
$$
C = \frac{n^2}{(n-s-1)^2}(1-\delta)^{-2} \frac{m}{n} \alpha^{-1} (\beta + \xi m \alpha^{-1})  + \alpha^{-1} \frac{m^2 (n-s+1)}{s-m+1} \exp(-\frac{c_{\delta}  \alpha (n-s-1) }{m}) 
$$
Taking $m \le s \le n/2 - 1$, we have 
$$
D \le 4 (1-\delta)^{-2} 
$$ 
and 
$$
C\le 4 \frac{m}{n}(1-\delta)^{-2}(\beta + \xi m \alpha^{-1}) + \alpha^{-1} m^2n  \exp(-\frac{c_{\delta}  \alpha n }{2m}) .$$
Therefore, provided $n \ge 2m + 2$ and 
$$n \ge 2 m \alpha^{-1} c_{\delta}^{-1} \log(\zeta^{-1} m^2 n ),$$ 
it holds 
$$
C \le 4 \frac{m}{n}(1-\delta)^{-2}(\beta + \xi m \alpha^{-1}) + \alpha^{-1} \zeta .
$$ 
\end{proof}

\begin{lemma}\label{lem:dpp_exp_tr}
Let $\x \sim \gamma^\nu_n$ with $\nu=\revb{w^{-1}}\mu$ satisfying $\revb{w^{-1}} \ge \alpha \revb{w_m^{-1}}$. Then for any function $f$, it holds 
$$
\Ebb(f(x_1) \trace((\Phi^{w}(\x)^T\Phi^{w}(\x))^{-1})) \le ( \frac{m}{n}  + \alpha^{-1} \frac{m(m-1)}{n(n-m+1)} )\Ebb_{x\sim \nu}(f(x)),
$$
with an equality if $\Phi^{w}(\y)$ is almost surely of rank $m$ for $\y \sim \nu^{\otimes n}$. In the particular case where $\nu=\nu_m$, it holds 
$$
\Ebb(f(x_1) \trace((\Phi^{w_m}(\x)^T\Phi^{w_m}(\x))^{-1})) \le  \frac{m}{n-m+1} \Ebb_{x\sim \nu}(f(x)),
$$
with an equality if $\Phi^{w_m}(\x)$ is almost surely of rank $m$ for $\y \sim \nu_m^{\otimes m}.$
\end{lemma}
\begin{proof}
The proof follows the one of \cite[Lemma 3.4]{derezinski2022unbiased} \reva{for the particular case $\nu = \nu_m$}. \reva{For completeness, we here detail the proof for a general case.}
Letting $\y \sim \nu^{\otimes n}$ and $\A(\x) := \Phi^{w}(\x)$, we have 
\begin{align*}
\Ebb(f(x_1) \trace((\Phi^{w}(\x)^T\Phi^{w}(\x))^{-1})) &= \Ebb(f(x_1) \trace((\A(\x)^T\A(\x))^{\dagger}))
 \\
 &\le \frac{(n-m)!}{n!}\Ebb(f(y_1) \trace( \adj(\A(\y)^T\A(\y))) ), 
\end{align*}
where the inequality becomes an equality when $\Phi^{w}(\y)$ is almost surely full rank. From the Cauchy-Binet formula, we have
\begin{align*}
\Ebb(f(y_1) \trace( \adj(\A(\y)^T\A(\y)))) = \frac{1}{n-m+1}  \sum_{i=1}^n \Ebb(f(y_1) \trace( \adj(\A(\y_{-i})^T\A(\y_{-i}))))& \\
= \frac{1}{n-m+1}  \Ebb(f(y_1)) \Ebb(\trace( \adj(\A(\z)^T\A(\z))))  + 
  \frac{n-1}{n-m+1}  \Ebb(f(z_1) \trace( \adj(\A(\z)^T\A(\z))))
\end{align*}
with $\z \sim \nu^{\otimes (n-1)}.$ Using Lemma \ref{lem:expectation_det}, we have 
\begin{align*}
 \Ebb(\trace( \adj(\A(\z)^T\A(\z)))) &= \frac{(n-1)!}{(n-1)^{m-1} (n-m)!}\trace( \adj( \Ebb(\A(\z)^T\A(\z)))) \\
 &= \frac{(n-1)!}{(n-1)^{m-1} (n-m)!}\trace( \adj((n-1)\I_{m}) ) \\
 &= \frac{(n-1)!}{(n-1)^{m-1} (n-m)!}\trace((n-1)^{m-1} \I_{m} )  \\
 &= \frac{(n-1)!}{ (n-m)!}  m.
\end{align*}
Letting $B_n := \Ebb_{\y \sim \nu^{\otimes n}}(f(y_1) \trace( \adj(\A(\y)^T\A(\y))))$, we have found 
\begin{align*}
B_n &= \frac{(n-1)!}{ (n-m +1)!}  m \Ebb_{x\sim \nu}(f(x)) +  B_{n-1}  = \frac{(n-1)!}{ (n-m)!}  m \Ebb_{x\sim \nu}(f(x))  + \binom{n-1}{m-2}B_{m-1}
\end{align*}
where the last equality is obtained by induction. It remains to evaluate $B_{m-1}.$
Let now $\z \sim \nu^{\otimes (m-1)}$. Letting $\A^{-j}(\z) $ being the matrix 
$\A(\z)$ without the $j$-th column, and letting 
$\a^{-j}(z_1)$ being the first row vector of $\A^{-j}(\z)$, that is the vector $\varphib^{w}(z_1)$ without the $j$-th entry, we have 
\begin{align*}
\trace( \adj(\A(\z)^T\A(\z))))
 &= \sum_{j=1}^m \adj(\A(\z)^T\A(\z))_{jj}= \sum_{j=1}^m \det(\A^{-j}(\z)^T\A^{-j}(\z)) \\
&= \sum_{j=1}^m \det(\A^{-j}(\z_{-1})^T\A^{-j}(\z_{-1}) + \a^{-j}(z_1)\a^{-j}(z_1)^T)\\
&= \sum_{j=1}^m \a^{-j}(z_1)^T \adj(\A^{-j}(\z_{-1})^T\A^{-j}(\z_{-1})) \a^{-j}(z_1) ,
\end{align*}
where we have used $\det(\A^{-j}(\z_{-1})^T\A^{-j}(\z_{-1}) ) = 0$. Then using Lemma \ref{lem:expectation_det}, we obtain 
 \begin{align*}
B_{m-1}&=\Ebb(f(z_1) \trace( \adj(\A(\z)^T\A(\z)))))\\
 &=  \sum_{j=1}^m \Ebb(f(z_1)\a^{-j}(z_1)^T \Ebb(\adj(\A^{-j}(\z_{-1})^T\A^{-j}(\z_{-1}))) \a^{-j}(z_1))\\
&= \frac{(m-2)!}{(m-2)^{m-2}} \sum_{j=1}^m \Ebb(f(z_1)\a^{-j}(z_1)^T \adj(\Ebb(\A^{-j}(\z_{-1})^T\A^{-j}(\z_{-1}))) \a^{-j}(z_1))\\
&= \frac{(m-2)!}{(m-2)^{m-2}} \sum_{j=1}^m \Ebb(f(z_1) \a^{-j}(z_1)^T \adj((m-2) \I_{m-1}) \a^{-j}(z_1))\\
&= (m-2)! \sum_{j=1}^m \Ebb(f(z_1) \Vert \a^{-j}(z_1) \Vert_2^2) \\
&= (m-2)! \sum_{j=1}^m \Ebb(f(z_1) (\Vert \varphib^{w}(z_1) \Vert_2^2 - \varphi_j^{w}(z_1)^2))\\
&= (m-1)! \,  \Ebb(f(z_1) \Vert \varphib^{w}(z_1) \Vert_2^2) \\
&= (m-1)! \,  \Ebb(f(z_1) \revb{w(z_1) m w_m(z_1)^{-1}}) \\
&\le m!  \alpha^{-1}   \Ebb(f(z_1)) 
\end{align*}
where we have used $\Vert \varphib^{w}(z) \Vert_2^2 = \revb{w(z)} \Vert \varphib(z) \Vert_2^2 \le \alpha^{-1} \revb{w_m(z)} \Vert \varphib(z) \Vert_2^2 = m$.
We finally obtain 
\begin{align*}
\Ebb(f(x_1) \trace((\Phi^{w}(\x)^T\Phi^{w}(\x))^{-1})) &\le \frac{m}{n} \Ebb_{x\sim \nu}(f(x)) \\
&\quad +  \alpha^{-1} \frac{(n-m)!}{n!} m!  \binom{n-1}{m-2} \Ebb_{x\sim \nu}( f(x)   )\\
& =( \frac{m}{n}  + \alpha^{-1} \frac{m(m-1)}{n(n-m+1)} )\Ebb_{x\sim \nu}(f(x)).
\end{align*}

\end{proof}

\bibliographystyle{plain}

\end{document}